\documentclass[10pt]{amsart}
\usepackage{amsmath,amsthm,amssymb,xypic,array}
\usepackage[T1]{fontenc}
\usepackage{amsfonts}
\usepackage{amsmath}
\usepackage{amssymb}
\usepackage{amsthm}
\usepackage{setspace}
\usepackage[cp850]{inputenc}
\usepackage{hyperref}
\usepackage{graphicx}
\usepackage{vmargin}
\usepackage{fancyhdr}
\usepackage{tikz}
\usepackage{booktabs}
\usepackage{longtable}
\usepackage{geometry}
\usetikzlibrary{trees}

\setpapersize{A4}
\providecommand{\U}[1]{\protect\rule{.1in}{.1in}}
\providecommand{\U}[1]{\protect\rule{.1in}{.1in}}
\providecommand{\U}[1]{\protect\rule{.1in}{.1in}}
\providecommand{\U}[1]{\protect\rule{.1in}{.1in}}
\providecommand{\U}[1]{\protect\rule{.1in}{.1in}}
\input{xy}
\xyoption{all}
\theoremstyle{theorem}

\newtheorem{Theorem}{Theorem}[section]
\newtheorem{theoremn}{Theorem}

\newtheorem*{conjecturen}{Conjecture}
\newtheorem{Lemma}[Theorem]{Lemma}
\newtheorem{Proposition}[Theorem]{Proposition}
\newtheorem{Corollary}[Theorem]{Corollary}

\theoremstyle{definition}

\newtheorem{Definition}[Theorem]{Definition}
\newtheorem{Remark}[Theorem]{Remark}

\numberwithin{equation}{section}
\newcommand{\arXiv}[1]{\href{http://arxiv.org/abs/#1}{arXiv:#1}}
\newcommand{\Spec}{\operatorname{Spec}}
\newcommand{\logpol}{\operatorname{log}}
\newcommand{\Aut}{\operatorname{Aut}}
\newcommand{\id}{\operatorname{Id}}
\newcommand{\Exts}{\operatorname{{\mathcal E}xt}}
\newcommand{\ch}{\operatorname{char}}
\newcommand{\eps}{\varepsilon}

\newcommand{\Ext}{\operatorname{Ext}}

\newcommand{\im}{\operatorname{Im}}
\newcommand{\Sym}{\operatorname{Sym}}
\newcommand{\Pic}{\operatorname{Pic}}
\newcommand{\Proj}{\operatorname{Proj}}
\newcommand{\fSpec}{\operatorname{Spf}}
\newcommand{\codim}{\operatorname{codim}}
\newcommand{\Sing}{\operatorname{Sing}}

\def\bibaut#1{{\sc #1}}

\begin{document}

\title{On the rigidity of moduli of curves in arbitrary characteristic}

\author[Barbara Fantechi]{Barbara Fantechi}
\address{\sc Barbara Fantechi\\
SISSA\\
via Bonomea 265\\
34136 Trieste\\ Italy}
\email{fantechi@sissa.it}

\author[Alex Massarenti]{Alex Massarenti}
\address{\sc Alex Massarenti\\
IMPA\\
Estrada Dona Castorina 110\\
22460-320 Rio de Janeiro\\ Brazil}
\email{massaren@impa.br}

\date{\today}

\maketitle

\begin{abstract}
The stack $\overline{\mathcal{M}}_{g,n}$ of stable curves and its coarse moduli space $\overline{M}_{g,n}$ are defined over $\mathbb{Z}$, and therefore over any field. Over an algebraically closed field of characteristic zero, in \cite{Hac} Hacking showed that $\overline{\mathcal{M}}_{g,n}$ is rigid (a conjecture of Kapranov), in \cite{BM}, \cite{Ma} Bruno and Mella for $g=0$, and the second author for $g\geq 1$ showed that its automorphism group is the symmetric group $S_n$, permuting marked points unless $(g,n)\in\{(0,4),(1,1),(1,2)\}$.\\
The methods used in the papers above do not extend to positive characteristic. We show that in characteristic $p>0$, the rigidity of $\overline{\mathcal{M}}_{g,n}$, with the same exceptions as over $\mathbb C$, implies that its automorphism group is $S_n$. We prove that, over any perfect field, $\overline{M}_{0,n}$ is rigid and deduce that, over any field, $\Aut(\overline{M}_{0,n})\cong S_{n}$ for $n\geq 5$.\\ 
Going back to characteristic zero, we prove that for $g+n>4$, the coarse moduli space $\overline M_{g,n}$ is rigid, extending a result of Hacking who had proven it has no locally trivial deformations.  Finally, we show that $\overline{M}_{1,2}$ is not rigid, although it does not admit locally trivial deformations,  by explicitly computing his Kuranishi family.
\end{abstract}

\setcounter{tocdepth}{1}
\tableofcontents

\section*{Introduction}
The stack $\overline{\mathcal{M}}_{g,n}$ parametrizing Deligne-Mumford stable curves and its coarse moduli space $\overline{M}_{g,n}$ are among the most fascinating and widely studied objects in algebraic geometry. A remarkable property of $\overline{\mathcal{M}}_{g,n}$ is that it is smooth and proper over $\Spec(\mathbb{Z})$, and thus $\overline{\mathcal{M}}_{g,n}^R$ is defined over any commutative ring $R$ via base change
 \[
  \begin{tikzpicture}[xscale=2.5,yscale=-1.2]
    \node (A0_0) at (0, 0) {$\overline{\mathcal{M}}_{g,n}^{R}$};
    \node (A0_1) at (1, 0) {$\overline{\mathcal{M}}_{g,n}$};
    \node (A1_0) at (0, 1) {$\Spec(R)$};
    \node (A1_1) at (1, 1) {$\Spec(\mathbb{Z})$};
    \path (A0_0) edge [->] node [auto] {$\scriptstyle{}$} (A0_1);
    \path (A1_0) edge [->] node [auto] {$\scriptstyle{}$} (A1_1);
    \path (A0_1) edge [->] node [auto] {$\scriptstyle{}$} (A1_1);
    \path (A0_0) edge [->] node [auto] {$\scriptstyle{}$} (A1_0);
  \end{tikzpicture}
  \]
By \cite{KM} the formation of the coarse moduli space is compatible with flat base change; we write $\overline{M}_{g,n}^R$ for the coarse moduli scheme of $\overline{\mathcal{M}}_{g,n}^R$.

The symmetric group $S_n$ acts via permutations of the marked points on $\overline{M}_{g,n}^R$ and $\overline{\mathcal{M}}_{g,n}^R$ for every ring $R$. The biregular automorphisms of the moduli space $M_{g,n}$, and of its Deligne-Mumford compactification $\overline{M}_{g,n}$ have been studied in a series of papers \cite{BM}, \cite{Ro}, \cite{Li1}, \cite{Li2}, \cite{GKM} and \cite{Ma}.\\
It is known that, with a short and explicit list of exceptions, over the field of complex numbers both the automorphism groups of the stack and of the coarse moduli space are isomorphic to $S_n$, see Appendix \ref{app} for details.

In Section \ref{autrig} we first study how automorphisms of a scheme behave with respect to field extensions, and then, given a scheme $X\rightarrow\Spec(A)$ over a local ring $A$ with residue filed $K$ and generic point $\xi\in\Spec(A)$, we show how the infinitesimal rigidity of $X^K$ plays a fundamental role in lifting automorphisms of $X^K$ to automorphisms of $X_{\xi}$.

Finally, in Section \ref{aaf} we apply these results together with the rigidity results in Section \ref{g0} to $\overline{M}_{0,n}^K$, with $K$ a field of positive characteristic and $A = W(K)$ the ring of Witt vector over $K$, in order to compute the automorphism group of $\overline{M}_{0,n}^K$. The main results on the automorphism groups in Proposition \ref{autm12}, Theorem \ref{aut1} and Appendix \ref{app} can be summarized as follows:

\begin{theoremn}
Let $K$ be any field. Then 
$$\Aut(\overline{M}_{0,n}^{K})\cong S_n$$ 
for any $n\geq 5$, and $\Aut(\overline{M}_{0,4}^K)\cong PGL(2)$. If $\ch(K)\neq 2$, $2g-2+n\geq 3$, $(g,n)\neq (2,1)$ and $n\geq 1$ then 
$$\Aut(\overline{\mathcal{M}}_{g,n}^{K})\cong\Aut(\overline{M}_{g,n}^{K})\cong S_n.$$
Over any field $K$ of characteristic zero we have that $\Aut(\overline{\mathcal{M}}_{g}^{K})\cong\Aut(\overline{M}_{g}^{K})$ for any $g\geq 2$, and $\Aut(\overline{\mathcal{M}}_{2,1}^{K})\cong\Aut(\overline{M}_{2,1}^{K})$ are trivial .\\
Finally if $K$ is an algebraically closed field with $\ch(K)\neq 2,3$ we have
$$\Aut(\overline{M}_{1,2}^K)\cong (K^{*})^2$$
while $\Aut(\overline{\mathcal{M}}_{1,2}^K)$ is trivial, $\Aut(\overline{\mathcal{M}}_{1,1}^K)\cong K^{*}$, and $\Aut(\overline{M}_{1,1}^K)\cong PGL(2)$.
\end{theoremn}

The rigidity of the moduli stack of stable curves is interesting in its own sake, as a special case of the following conjecture proposed by M. Kapranov in  \cite{Ka}:
\begin{conjecturen}
Let $X$ be a smooth scheme, and let $X_{1}$ be the moduli space of deformations of $X$, $X_{2}$ the moduli space of deformations of $X_{1}$, and so on. Then $X_{\dim(X)}$ is rigid, that is it does not have infinitesimal deformations. 
\end{conjecturen}
Over a field of characteristic zero the rigidity of the stack $\overline{\mathcal{M}}_{g,n}$ has been proven by P. Hacking in \cite[Theorem 2.1]{Hac}. Furthermore, by \cite[Theorems 2.3]{Hac}, the coarse moduli space $\overline{M}_{g,n}$ does not have non-trivial, locally trivial deformations if $(g,n)\notin\{(1,2),(2,0),(2,1),(3,0)\}$. 

In Section \ref{hacking} we review Hacking's proof and explain why it fails to extend to positive characteristic; the main problem is the lack of a suitable version of Kodaira vanishing. We then show that the requested vanishing can be proven directly in genus zero, thus establishing the rigidity of $\overline M_{0,n}^K=\overline {\mathcal M}_{0,n}^K$
for any field $K$, and hence deducing that its automorphism group if $S_n$ for $n\geq 5$.

Finally, in Section \ref{rcms} we address a question Hacking left open, namely the rigidity of the scheme $\overline M_{g,n}^{K}$ when $K$ is an algebraically closed field of characteristic zero. We show that the Kuranishi family of $\overline M_{1,2}^{K}$ is non-singular of dimension six and that the general deformation is smooth; for $g+n>4$ we show that $\overline M_{g,n}^{K}$ is rigid. This provides an interesting example of a rigid scheme whose singularities are not locally rigid. The main result in \cite{Hac} together with Theorems \ref{rig}, \ref{loctriv}, \ref{nonrig}, \ref{rigcms} can be summarized in the following statement:
\begin{theoremn}
Over any field $K$ the moduli space $\overline{M}_{0,n}^{K}$ is rigid for any $n\geq 3$. 
If $K$ is an algebraically closed field of characteristic zero then the stack $\overline{\mathcal{M}}_{g,n}^K$ is rigid, and its coarse moduli space $\overline{M}_{g,n}^K$ is rigid if $g+n>4$ and does not have locally trivial deformations if $g+n\geq 4$.\\
Finally, the coarse moduli space $\overline{M}_{1,2}^K$ does not have locally trivial deformations, while its family of first order infinitesimal deformations is non-singular of dimension six and the general deformation is smooth.   
\end{theoremn}
Finally, we would like to mention that recently the non-commutative rigidity of $\overline{\mathcal{M}}_{g,n}$, in characteristic zero, has been proven by means of the vanishing of the second Hochschild cohomology by S. Okawa and T. Sano in \cite{OS}.   

\section{Automorphisms and rigidity}\label{autrig}

First of all, we need some elementary results on the behavior of automorphisms with respect to field extensions. In the following by orbifold we always mean a non-singular Deligne-Mumford stack with trivial general stabilizer; this assumption implies that the automorphism groupoid is rigid, hence can be identified with a group.
  
\begin{Remark}\label{l1}
Let $K\subseteq L$ be a field extension, $X_{K}$ a scheme, or an orbifold, over $K$ and $X_L := X_{K}\times_{\Spec(K)}\Spec(L)$. Any morphism $f_{K}:X_{K}\rightarrow X_{K}$ over $K$ induces a unique morphism $f_{L}:X_{L}\rightarrow X_{L}$ over $L$ by the universal property of the fiber product applied to the following commutative diagram.
\[
  \begin{tikzpicture}[xscale=1.5,yscale=-1.2]
    \node (A0_0) at (0, 0) {$X_{L}$};
    \node (A0_2) at (2, 0) {$X_{L}$};
    \node (A1_3) at (3, 1) {$X_{K}$};
    \node (A1_5) at (5, 1) {$X_{K}$};
    \node (A2_1) at (1, 2) {$\Spec(L)$};
    \node (A2_4) at (4, 2) {$\Spec(K)$};
    \path (A0_2) edge [->] node [auto] {$\scriptstyle{}$} (A1_5);
    \path (A0_0) edge [->] node [auto] {$\scriptstyle{}$} (A1_3);
    \path (A1_5) edge [->] node [auto] {$\scriptstyle{}$} (A2_4);
    \path (A0_0) edge [->,dashed] node [auto] {$\scriptstyle{f_{L}}$} (A0_2);
    \path (A2_1) edge [->] node [auto] {$\scriptstyle{}$} (A2_4);
    \path (A1_3) edge [->] node [auto] {$\scriptstyle{}$} (A2_4);
    \path (A0_2) edge [->] node [auto] {$\scriptstyle{}$} (A2_1);
    \path (A1_3) edge [->] node [auto] {$\scriptstyle{f_{K}}$} (A1_5);
    \path (A0_0) edge [->] node [auto] {$\scriptstyle{}$} (A2_1);
  \end{tikzpicture}
  \]
\end{Remark}

\begin{Lemma}\label{c1}
Let $K\subseteq L$ be a field extension, $X_{K}$ a scheme over $K$ and $X_L := X_{K}\times_{\Spec(K)}\Spec(L)$. Then there exists an injective morphism of groups
$$\chi:\Aut(X_{K})\rightarrow\Aut(X_{L})$$
\end{Lemma}
\begin{proof} The morphism is an immediate consequence of Remark \ref{l1}. To prove injectivity,
let $\phi_{K}\in \Aut(X_{K})$ be an automorphism such that the induced automorphism $\phi_{L}\in\Aut(X_{L})$ is the identity; we need to show the $\phi_K$ is also the identity. Being the identity is local in the \'etale topology, so we may assume $X=\Spec(A)$ with $A$ a $K$-algebra and $\phi_K$ induced by $f\in \Aut(A)$. We have that $f-\id_A$ is a $K$-linear automorphism of $A$ which is zero when tensored with $L$, therefore it is zero.
\end{proof}

In the rest of this section $A$ will be a local ring with maximal ideal $\mathfrak{m}$ and residue field $K$, which is complete in the $\mathfrak{m}$-adic topology. 
Let $X$ be a Deligne-Mumford stack and $f:X\to \Spec(A)$ be a proper, flat morphism. Let $A_n:=A/\mathfrak{m}^{n+1}$, and $X_n:=X\times_{\Spec(A)}\Spec(A_n)$.  Write $\Aut(X_n/A_n)$ for the automorphism group of $X_n$ over $\Spec(A_n)$. Let $\xi$ be the generic point of $\Spec(A)$ and $X_\xi$ the corresponding fiber of $X$. Assume that $X_0$ is either a scheme or an orbifold.

\begin{Lemma}\label{liftaut} Assume that $X_0$ is smooth, then for every $n\ge 0$ there is an exact sequence of sets
$$0\rightarrow H^{0}(X_{0},T_{X_{0}})\otimes \mathfrak{m}^{n}/\mathfrak{m}^{n+1}\rightarrow \Aut(X_{n+1})\rightarrow\Aut(X_{n})\rightarrow H^{1}(X_{0},T_{X_{0}})\otimes \mathfrak{m}^{n}/\mathfrak{m}^{n+1}$$
where $\alpha:\Aut(X_{n+1})\rightarrow\Aut(X_{n})$ is the restriction map. In particular, if $H^i(X_0,T_{X_0})=0$ for $i=0,1$, then $\Aut(X/A)$ is formally \'etale over $A$.
\end{Lemma}
\begin{proof} By exact sequence we mean that an element $f\in \Aut(X_n/A_n)$ is in the image of $\alpha$ if and only if its image in $H^{1}(X_{0},T_{X_{0}})\otimes \mathfrak{m}^{n}/\mathfrak{m}^{n+1}$ is zero, and in this case the set of inverse images is naturally a principal homogeneous space for the group $H^{0}(X_{0},T_{X_{0}})\otimes \mathfrak{m}^{n}/\mathfrak{m}^{n+1}$. The statement follows from \cite[Proposition III. 2.2.2]{Ill} by observing that since $X$ is smooth over $\Spec(A)$ the cotangent complex is $L_{X/A}=\Omega_{X/A}$.
\end{proof}

\begin{Remark} If we only assume that $X_0$ is reduced Lemma \ref{liftaut} still hold with $\Ext^1(\Omega_{X_0},\mathcal O_{X_0})$ in place of $H^{1}(X_{0},T_{X_{0}})$; this  is an easy and well-known computation. 
\end{Remark}

\begin{Definition} Let $X$ be a scheme over a field. We will say it is {\em rigid} if it has no non-trivial infinitesimal deformations; if $X$ is smooth, this is equivalent to $H^1(X,T_X)=0$ and if $X$ is generically reduced this is equivalent to $\Ext^1(\Omega_X,\mathcal O_X)=0$.
\end{Definition}

\begin{Theorem}\label{lift}
Assume that $X\rightarrow \Spec(A)$ is a proper, separated morphism of schemes, that $X_0\rightarrow \Spec(A/\mathfrak{m})$ is generically reduced and rigid, and that $H^0(X_\xi, T_{X_\xi})=0$. Then there is a natural injective morphism of groups
$$\chi: \Aut(X_0)\rightarrow \Aut(X_{\xi}).$$
\end{Theorem}
\begin{proof}
We have a family of infinitesimal deformations $\{X_n\}_{n\geq 0}$ where $X_n$ is a scheme over $A_n$. Since $X_0$ is rigid, by semi-continuity $H^0(X_\xi,T_{X_\xi})=0$ implies $H^0(X_0,T_{X_0})=0$. \\
Lemma \ref{liftaut} implies that any automorphism of $X_n$ lifts to a unique automorphism of $X_{n+1}$. Therefore, starting with an automorphism $\phi:=\phi_0\in\Aut(X_0)$ we may construct a family of automorphisms $\phi_{\bullet} = \{\phi_{n}\}_{n\geq 0}$, where $\phi_n$ is an automorphism of $X_n$. The direct limit 
$$\widehat{X} = \varinjlim_{n}X_{n}\rightarrow\fSpec(A)$$ 
is a formal scheme over the formal spectrum $\fSpec(A)$ of $A$. By \cite[Section 8.1.5]{FAG} $\phi_{\bullet}$ induces a unique automorphism $\widehat{\phi}:\widehat{X}\rightarrow\widehat{X}$, and by \cite[Corollary 8.4.7]{FAG} $\widehat{\phi}$ in turn induces an unique automorphism $\overline{\phi}$ of $X\rightarrow\Spec(A)$. Let $X_{\xi}$ be the general fiber of $X\rightarrow\Spec(A)$. By taking the restriction of $\overline{\phi}$ to $X_{\xi}$ we get a morphism of groups  
$$
\begin{array}{cccc}
\chi: & \Aut(X_0) & \longrightarrow & \Aut(X_{\xi})\\
 & \phi & \longmapsto & \overline{\phi}_{|X_{\xi}}
\end{array}
$$
Clearly if $\phi$ and $\psi$ are two automorphisms of $X_0$ inducing the same automorphism of $X_{\xi}$ then $\overline{\phi} = \overline{\psi}$ because $\xi$ is the general point of $\Spec(A)$ and $X$ is separated. In particular $\overline{\phi}$ and $\overline{\psi}$ coincide on the special fiber $X_0$, that is $\phi = \psi$. So the morphism $\chi$ is injective.
\end{proof}

We will need to following result in order to study the automorphism group of the stack $\overline{\mathcal{M}}_{g,n}$.

\begin{Proposition}\label{autstack}
Let $\mathcal{X}$ be a separated orbifold over an affine scheme $S$ and let $X$ be its coarse moduli space. Then there exists an injective morphism of groups
$$\chi:\Aut(\mathcal{X}/S)\rightarrow\Aut(X/S).$$
\end{Proposition}
\begin{proof}
Since the natural map $\pi:\mathcal{X}\rightarrow X$ is universal for morphisms in schemes for any $\phi\in \Aut(\mathcal{X})$ we get a unique $\tilde{\phi}\in \Aut(X)$ commuting with $\pi$. This correspondence induces the morphism $\chi$. The orbifold condition implies that there is an open substack $U$ of $\mathcal X$ which is isomorphic to its image in $X$; if $\phi\in \Aut(X/S)$ satisfies $\chi(\phi)=\id_X$ it follows that $\phi|_U=\id_U$ and by separateness this implies that $\phi=\id_{\mathcal X}$.
\end{proof}

\section{A review of Hacking's proof of rigidity in characteristic zero}\label{hacking}

We begin this section working over the integers. We will point out when we pass to characteristic zero. Let $X$ be a smooth algebraic stack, $i:D\to X$ the closed embedding of a normal crossing divisor, nc for short, $\nu:D^\nu\to D$ its normalization; by definition of nc, $D^\nu$ is smooth.\\
The composition $\bar\nu:=i\circ\nu:D^\nu\to X$ is an immersion, that is the differential has maximum rank at every point, hence it has a normal line bundle $\mathcal N$, defined by the exact sequence $$
0\mapsto T_{D^\nu}\to \bar\nu^*(T_X)\to \mathcal N\to 0$$
The vector bundle of logarithmic differentials $T_X(-\log D)$ is defined by the exact sequence $$
0\mapsto T_X(-\log D)\to T_X\to \bar\nu_*(\mathcal N)\to 0$$
Let $\pi:\mathcal U_{g,n}\to \overline{\mathcal M}_{g,n}$ be the universal curve. Let $\mathcal B$ be the boundary of $\overline{\mathcal M}_{g,n}$; it is a normal crossing divisor. Its normalization $\mathcal B^\nu$ is the disjoint union of products $\overline{\mathcal M}_{g_1,n_1+1}\times \overline{\mathcal M}_{g_2,n_2+1}$ with $g_1+g_2=g$ and $n_1+n_2=n$, together with $\overline{\mathcal M}_{g-1,n+2}$ by \cite[Corollary 3.9]{Kn2}. The restriction of $\mathcal N$ to each of these components is $p_1^*\psi^\vee_{n_1+1}\otimes p_2^*\psi^\vee_{n_2+1}$ and $\psi^\vee_{n+1}\otimes\psi^\vee_{n+2}$, respectively.
Let $\Sigma\subset \mathcal U_{g,n}$ be the union of the images of the $n$ sections of $\pi$, this is an nc divisor.

\begin{Lemma}\cite[Lemma 3.1]{Hac} There is an isomorphism $T_{\overline{\mathcal M}_{g,n}}(-\log \mathcal B)\to R^1\pi_*(\omega_\pi(\Sigma)^\vee)$.
\end{Lemma}
\begin{Corollary} For every $k\ge 0$ we have $H^k(\mathcal U_{g,n}\omega_\pi(\Sigma)^\vee)=H^{k-1}(\overline{\mathcal M}_{g,n}, T_{\overline{\mathcal M}_{g,n}}(-\log \mathcal B))$.
\end{Corollary}
\begin{proof} Since all other push-forwards of $\omega_{\pi}(\Sigma)^\vee$ vanish, it follows from Leray's spectral sequence.
\end{proof}

\begin{Lemma}\cite[Theorem 4]{Kn2}\label{psi} The pull-back under the natural isomorphism $\overline{\mathcal M}_{g,n+1}\to \mathcal U_{g,n}$ of the line bundle $\omega_\pi(\Sigma)$ is $\psi_{n+1}$.
\end{Lemma}

Putting these results together, Hacking is able to reduce the rigidity of $\overline{\mathcal M}_{g,n}$ to vanishing of cohomologies of line bundles.

\begin{Proposition}\label{reducetolinebundles} Assume that $H^2(\overline{\mathcal M}_{g,n+1}, \psi_{n+1}^\vee)=0$, and that $H^1(\mathcal B^\nu,\mathcal N)=0$. Then $\overline{\mathcal M}_{g,n}$ is rigid. More generally, if $H^{i+1}(\overline{\mathcal M}_{g,n+1}, \psi_{n+1}^\vee)=H^i(\mathcal B^\nu,\mathcal N)=0$ then $H^i(\overline{\mathcal M}_{g,n}, T_{\overline{\mathcal M}_{g,n}})=0$.
\end{Proposition}

\subsection{Characteristic zero part}

By \cite[Theorem 0.4]{Ke} the $\mathbb{Q}$-line bundle $p_{*}\omega_{\pi}(\Sigma)$ is nef and big. In \cite[Theorem 3.2]{Hac} Hacking proves that, over a field $K$ of characteristic zero
\begin{equation}\label{van}
H^{i+1}(\mathcal{U}_{g,n},\omega_{\pi}(\Sigma)^{\vee})\cong H^{i}(\overline{\mathcal{M}}_{g,n},T_{\overline{\mathcal{M}}_{g,n}}(-\logpol(\mathcal{B})) = 0
\end{equation}
for any $i < \dim(\overline{\mathcal{M}}_{g,n})$. Furthermore, by \cite[Corollary 4.4]{Hac} the $\mathbb{Q}$-line bundle on the coarse moduli space of $\mathcal{B}^{\nu}$ defined by $\mathcal{N}^{\vee}$ is nef and big on each component. Then
\begin{equation}\label{van2}
H^{i}(\overline{\mathcal{M}}_{g,n},\nu_{*}\mathcal{N}) = 0
\end{equation}
for $i < \dim(\mathcal{B})$. Finally, combining the vanishings (\ref{van}) and (\ref{van2}) Hacking proves that 
$$H^{i}(\overline{\mathcal{M}}_{g,n},T_{\overline{\mathcal{M}}_{g,n}}) = 0$$ 
for $i < \dim(\overline{\mathcal{M}}_{g,n})-1$. In particular $H^{1}(\overline{\mathcal{M}}_{g,n},T_{\overline{\mathcal{M}}_{g,n}}) = 0$, that is $\overline{\mathcal{M}}_{g,n}$ is rigid. Both the vanishings (\ref{van}) and (\ref{van2}) are consequences of \cite[Theorem A.1]{Hac} which is a version of Kodaira vanishing for proper and smooth Deligne-Mumford stacks. This theorem derives from the Kodaira vanishing theorem for proper normal varieties. Therefore it holds only in characteristic zero. At the best of our knowledge, the closest result to Kodaira vanishing in positive characteristic is the Deligne-Illusie vanishing \cite{DI}. However, such result works for ample line bundle while in positive characteristic the involved psi-classes are just semi-ample \cite{Ke}. 

\section{Rigidity of $\overline{M}_{0,n}$ in positive characteristic}\label{g0}

In this section we study the $g=0$ case; any $n$-pointed stable genus zero curve is automorphisms free, thus the coarse moduli space $\overline{M}_{0,n}$ is smooth and coincides with the stack. In \cite{Ka2} M. Kapranov constructed $\overline{M}_{0,n}$ via a sequence of blow-ups of $\mathbb{P}^{n-3}$ along linear spaces in order of increasing. The blow-up morphism $f:\overline{M}_{0,n}\rightarrow\mathbb{P}^{n-3}$ is induced by a psi-class $\psi_n$.\\ 
In \cite{Ka2} Kapranov works, for sake of simplicity, on an algebraically closed field of characteristic zero. On the other hand Kapranov's arguments are purely algebraic and his description works over $\Spec(\mathbb{Z})$.   

\begin{Theorem}\label{rig}
Let $K$ be a field, $n\ge 3$. Then, except for $i=0$ and $n=4$, we have
$$H^{i}(\overline{M}_{0,n},T_{\overline{M}_{0,n}}) = 0 \quad for \quad all \quad  \textit{i}$$
Hence, $\overline{M}_{0,n}$ is rigid for any $n\geq 3$.
\end{Theorem}
\begin{proof}
The case $n=3$ is obvious since the tangent bundle is zero, hence we may assume $n\ge 4$.
By Proposition \ref{reducetolinebundles} it is enough to prove that:
\begin{itemize}
\item[-] $H^i(\mathcal U_{0,n+1},\omega_\pi(\Sigma)^\vee)=0,$
\item[-] for all $n_1+n_2=n$ with $n_1, n_2\ge 2$, for all $i_1+i_2=i$ one has $$H^{i_1}(\overline M_{0,n_1+1},\psi_{n_1}^\vee)\otimes H^{i_2+1}(\overline M_{0,n_1+1},\psi_{n_2+1}^\vee)=0$$
\end{itemize}
Let us consider the line bundle $\omega_{\pi}(\Sigma)$ over $\mathcal{U}_{0,n}\cong\overline{M}_{0,n+1}$. By Lemma \ref{psi} $\omega_{\pi}(\Sigma)$ identifies with $\psi_{n+1}$. Recall that we have a birational morphism
$$f:\overline{M}_{0,n+1}\rightarrow\mathbb{P}^{n-2}$$
such that $f^{*}\mathcal{O}_{\mathbb{P}^{n-2}}(1)\cong \omega_{\pi}(\Sigma)\cong\psi_{n+1}$ given by the Kapranov's blow-up construction. Now, by \cite[Theorem 5.1]{Har} we have that $H^{i}(\mathbb{P}^{n-2},\mathcal{O}_{\mathbb{P}^{n-2}}(-1)) = 0$ for $\textit{i}< \textit{n}-2$. Furthermore we have $H^{0}(\mathbb{P}^{n-2},\mathcal{O}_{\mathbb{P}^{n-2}}(-1)) = 0$, and by Serre duality $H^{n-2}(\mathbb{P}^{n-2},\mathcal{O}_{\mathbb{P}^{n-2}}(-1)) \cong H^{0}(\mathbb{P}^{n-2},\mathcal{O}_{\mathbb{P}^{n-2}}(2-n)) = 0$ for $n\geq 3$. Then $H^{i}(\mathbb{P}^{n-2},\mathcal{O}_{\mathbb{P}^{n-2}}(-1)) = 0$ for any \textit{i}. By the projection formula we get
$$f_{*}(\psi_{n+1}^{\vee}) = f_{*}(\mathcal{O}_{\overline{M}_{0,n+1}}\otimes f^{*}\mathcal{O}_{\mathbb{P}^{n-2}}(-1))\cong f_{*}(\mathcal{O}_{\overline{M}_{0,n+1}})\otimes\mathcal{O}_{\mathbb{P}^{n-2}}(-1)$$
Since $f$ is a birational morphism between smooth varieties we have $f_{*}(\mathcal{O}_{\overline{M}_{0,n+1}})\cong\mathcal{O}_{\mathbb{P}^{n-2}}$, hence $f_{*}(\psi_{n+1}^{\vee})\cong\mathcal{O}_{\mathbb{P}^{n-2}}(-1)$.\\ 
Moreover since $f$ is a sequence of blow-ups of smooth varieties with smooth centers, we have $R^{i}f_{*}\psi_{n+1}^{\vee} = 0$ for $i > 0$, and by the Leray spectral sequence we get $H^{i}(\overline{M}_{0,n+1},\psi_{n+1}^{\vee})\cong H^{i}(\mathbb{P}^{n-2},f_{*}\psi_{n+1}^{\vee})$ for any $i\geq 0$. Therefore
$$H^{i}(\overline{M}_{0,n+1},\psi_{n+1}^{\vee})\cong H^{i}(\mathbb{P}^{n-2},\mathcal{O}_{\mathbb{P}^{n-2}}(-1)) = 0 \quad \rm for \: any \: \textit{i}$$
Summing up we proved the following vanishings:
\begin{equation}\label{van1}
H^{i+1}(\overline{M}_{0,n+1},\psi_{n+1}^{\vee}) \cong H^{i+1}(\overline{M}_{0,n+1},\omega_{\pi}(\Sigma)^{\vee})\cong H^{i}(\overline{M}_{0,n},T_{\overline{M}_{0,n}}(-\logpol(\mathcal{B})) = 0
\end{equation}
for any $i$. Now, let us consider the dual of the normal bundle $\mathcal{N}^{\vee}$. By \cite[Lemma 4.2]{Hac} the pull-back of $\mathcal{N}^{\vee}$ to $\overline{M}_{0,S_{1}\cup\{n_{1}+1\}}\times\overline{M}_{0,S_{2}\cup\{n_{2}+1\}}$ is identified with $p_1^{*}\psi_{n_{1}+1}\otimes p_2^{*}\psi_{n_{2}+1}$. Note that since $g = 0$ we do not have the boundary divisor parametrizing irreducible nodal curves.\\
Recalling again the Kapranov's construction we have that the pull-back of $\mathcal{N}^{\vee}$ to $\overline{M}_{0,S_{1}\cup\{n_{1}+1\}}\times\overline{M}_{0,S_{2}\cup\{n_{2}+1\}}$ defines a birational morphism
$$g:=f_{n_{1}+1}\times f_{n_{2}+1}:\overline{M}_{0,S_{1}\cup\{n_{1}+1\}}\times\overline{M}_{0,S_{2}\cup\{n_{2}+1\}}\rightarrow\mathbb{P}^{s_1-2}\times\mathbb{P}^{s_2-2}$$
where $s_i = |S_{i}|$ for $i = 1,2$. Then $\mathcal{N} = g^{*}(\mathcal{O}(-1,-1))$, where 
$$\mathcal{O}(-1,-1):=\mathcal{O}_{\mathbb{P}^{s_1-2}}(-1)\boxtimes\mathcal{O}_{\mathbb{P}^{s_2-2}}(-1).$$
By K\"unneth formula we have 
$$H^{i}(\mathbb{P}^{s_1-2}\times\mathbb{P}^{s_2-2},\mathcal{O}(-1,-1)) = \bigoplus_{j+k=i} H^{j}(\mathbb{P}^{s_1-2},\mathcal{O}_{\mathbb{P}^{s_1-2}}(-1))\otimes H^{k}(\mathbb{P}^{s_2-2},\mathcal{O}_{\mathbb{P}^{s_2-2}}(-1))$$
Therefore we get $H^{i}(\mathbb{P}^{s_1-2}\times\mathbb{P}^{s_2-2},\mathcal{O}(-1,-1)) = 0$ for any $i$. Finally, applying to $\mathcal{N}^{\vee}$ the argument we used in the first part of the proof for $\omega_{\pi}(\Sigma)^{\vee}$ we get $H^i(\mathcal B^\nu,\mathcal N)=0$ for any $i$. Now, to conclude it is enough to apply Proposition \ref{reducetolinebundles}.
\end{proof}

As an immediate consequence of the proof of Theorem \ref{rig} we get the following result on the deformations of the pair $(\overline{M}_{0,n},\partial\overline{M}_{0,n})$, see \cite[Section 3.3.3]{Se} for the general theory of deformations of the pair.

\begin{Corollary}\label{loctri}
Let $K$ be any field, $i\ge 0$, $n\ge 3$. Then 
$$H^{i}(\overline{M}_{0,n},T_{\overline{M}_{0,n}}(-\logpol(\mathcal{B}))) = 0$$
In particular, the pair $(\overline{M}_{0,n},\partial\overline{M}_{0,n})$ is rigid for any $n\geq 3$.
\end{Corollary} 

\section{Deformations of $\overline{M}_{1,2}$}\label{elliptic}

In this section we prove that the moduli space $\overline{M}_{1,2}$ is not rigid. We work throughout over an algebraically closed field $K$ with $\ch(K)\neq 2,3$.

\subsection{Preliminaries}

Choose $n$ and $a_0,\ldots,a_n$ positive integers. We will denote by $\mathbb P(a_0,\ldots, a_n)$ the associated weighted projective space, defined as the quotient scheme of the action of $K^*$ on $\mathbb A^{n+1}_0:=\mathbb A^n\setminus\{0\}$ defined by 
$$\lambda\cdot(x_0,\ldots,x_n)=(\lambda^{a_0}x_0,\ldots,\lambda^{a_n}x_n)$$
The quotient ${\mathbb P}(a_0,\ldots,a_n)$ is a projective scheme with quotient singularities of type $\frac{1}{a_i}(a_0,...,\hat{a}_i,...,a_n)$ of the $i$-th fundamental point. 

Since we are in characteristic different from $2$ and $3$, $\overline{\mathcal M}_{1,1}$ can be identified, using the Weierstrass form of the elliptic curve, with the weighted projective stack $\mathbb{ P}(4,6)$ and $\overline{\mathcal M}_{1,2}$ is the stack quotient $[X/(K^*\times K^*)]$ where $X\subset \mathbb A^3_0\times \mathbb A^2_0$ is the closed subset defined by the equation 
$$X := Z(zy^{2}-x^{3}-axz^{2}-bz^{3})\subset\mathbb{A}^{3}_{0}\times \mathbb{A}^{2}_{0}$$
and the action is defined by
\begin{equation}\label{act12}
\begin{array}{ccc}
(K^{*}\times K^{*})\times X &  \longrightarrow & X\\
((\lambda,\xi),(x,y,z,a,b)) & \longmapsto & (\xi\lambda^{2}x,\xi\lambda^{3}y,\xi z,\lambda^{4}a,\lambda^{6}b)
\end{array}
\end{equation}
We denote by $E_4$ and $E_6$ the elliptic curves corresponding to the points $(1,0)$ and $(0,1)$ of $\overline{\mathcal M}_{1,1}$.

\begin{Lemma}\label{facile} 
Let $\mu_4$ act on $E_4$ via $\lambda\cdot(x,y,z)=(\lambda^2x, \lambda^3y, z)$; let $p_1=(0,1,0)$ and $p_2=(0,0,1)$ be the fixed points. Then $\lambda\in \mu_4$ acts on $H^1(E_4,T_{E_4}(-p_1))$ by multiplication by $\lambda^2$ and on $T_{p_1}E_4$ and $T_{p_2}E_4$ by multiplication by $\lambda$.\\
Let $\mu_6$ act on $E_6$ via $\lambda\cdot(x,y,z)=(\lambda^2x, \lambda^3y, z)$; let $p_1=(0,1,0)$ be the fixed point, and $p_2=(0,0,1)$ the fixed point of the induced action of $\mu_3$. Then $\mu_6$ acts on $H^1(E_6,T_{E_6}(-p_1))$ by multiplication by $\lambda^2$ and on $T_{p_1}E_6$ by multiplication by $\lambda$. Finally, $\varepsilon\in\mu_3$ acts on $T_{p_2}E_6$ by multiplication by $\varepsilon$.
\end{Lemma}
\begin{proof} We consider the $E_6$ case. The function $\frac{x}{y}$ is a local coordinate at the point $p_1$, hence the action of $\lambda$ on $\Omega_{p_1}E_6$ is given by multiplication by $\lambda^{-1}$, and therefore the action on the tangent space is given by $\lambda$.  There is a canonical isomorphism $H^1(E_6,T_{E^6}(-p_1))\to H^1(E_6,T_{E^6})$; the latter is $\mu_6$-equivariantly dual to $H^0(E_6,\Omega_{E_6}^{\otimes 2})$ and by what we have just proven the action on one forms is given by $\lambda^{-2}$. The other cases are similar.
\end{proof}

Let $(C,p)$ be a smooth or irreducible nodal elliptic curve, that is a point in $\overline{\mathcal M}_{1,1}$. Since the characteristic is different from $2$ and $3$ we can put $C$ in Weierstrass from, that is there exists $(a,b)\in \mathbb{A}^{2}_0$ such that $(C,p)$ is isomorphic to $(C_{a,b},[0:1:0])$, where
\begin{equation}\label{weier}
C_{a,b} = Z(zy^{2}-x^{3}-axz^{2}-bz^{3})\subset\mathbb{P}^{2}.
\end{equation}
This representation is called \textit{Weierstrass representation} of the elliptic curve. Moreover the set of isomorphisms between $(C_{a,b},[0:1:0])$ and $(C_{\bar a,\bar b},[0:1:0])$ can be identified with $\{\lambda\in K\,|\, (\bar a,\bar b)=(\lambda^4 a,\lambda^6 b)\}$ by associating to each such $\lambda$ the map $\phi_\lambda(x,y,z)=(\lambda^2 x, \lambda^3 y, z)$.
Using the representation (\ref{weier}) we have:
$$E_{4} := \{y^{2}z = x^{3}+xz^{2}\}\subset\mathbb{P}^{2},\quad E_{6} := \{y^{2}z = x^{3}+z^{3}\}\subset\mathbb{P}^{2}.$$
The rational Picard group of $\overline{M}_{1,2}$ is freely generated by the two boundary divisors \cite{Mor}: the divisors $\Delta_{irr}$ parametrizing irreducible nodal curves, and $\Delta_{1}$ parametrizing elliptic tails. These divisors are both smooth, rational curves. The boundary divisor $\Delta_{1}$ has negative self-intersection. In the following we give an explicit description of the contraction of $\Delta_{1}$.

\subsection{$\overline M_{1,2}$ as toric variety}
Assume $n$ is an integer prime to the characteristic of $K$. We denote by $\frac{1}{n}(a_1,a_2)$ any surface singularity which is \'etale locally isomorphic to $\Spec(K[t_1,t_2]^{\mu_n})$ where $\mu_n$ acts by $\lambda\cdot(x_1,x_2)=(\lambda^{a_1}x_1,\lambda^{a_2}x_2)$.

\begin{Proposition}\label{M12}
The moduli space $\overline{M}_{1,2}$ is a surface with four singular points. Two singular points lie in $M_{1,2}$, and are:
\begin{itemize}
\item[-] a singularity of type $\frac{1}{2}(1,1)$ at the image of $(0,0,1,1,0)\in X$,
representing an elliptic curve of Weierstrass representation $E_{4}$ with marked points $[0:1:0]$ and $[0:0:1]$;
\item[-] a singularity of type $\frac{1}{3}(1,2)$ at the image of $(0,1,1,0,1)\in X$,
 representing an elliptic curve of Weierstrass representation $E_{6}$ with marked points $[0:1:0]$ and $[0:1:1]$.
\end{itemize}
The remaining two singular points lie on the boundary divisor $\Delta_{1}$, and are:
\begin{itemize}
\item[-] a singularity of type $\frac{1}{3}(1,1)$ at the image of $(0,1,0,1,0)\in X$,
 representing a reducible curve whose irreducible components are an elliptic curve of type $E_{6}$ and a smooth rational curve connected by a node;
\item[-] a singularity of type $\frac{1}{2}(1,1)$ at the image of $(0,1,0,0,1)\in X$,
 representing a reducible curve whose irreducible components are an elliptic curve of type $E_{4}$ and a smooth rational curve connected by a node.
\end{itemize}
\end{Proposition}
\begin{proof}
Note that on $X$, $z=0 \Rightarrow x=0 \Rightarrow y\neq 0$. So $X$ is covered by the charts $\{z\neq 0\}$ and $\{y\neq 0\}$. Consider first the chart $\{z\neq 0\}$. Then we may take $\xi = 1$. On this chart $X$ is given by $\{y^{2} = x^{3}+ax+b\}$ so $b = y^{2}-x^{3}-ax$. We can take $(x,y,a)$ as coordinates, and the action of $K^{*}\times K^{*}$ is given by $(\lambda,x,y,a)\mapsto (\lambda^{2}x,\lambda^{3}y,\lambda^{4}a)$. The point $(0,0,0)$ is stabilized by $K^{*}\times K^{*}$, so does not produce any singularity. Since $(2,3) = (3,4) = 1$ the points $(x,y,a)$ such that $xy\neq 0$ or $ya\neq 0$ have trivial stabilizer.\\
If $y = 0$ the action is given by $(\lambda,x,a)\mapsto (\lambda^{2}x,\lambda^{4}a)$. If $x = 0$ then $a\neq 0$, the stabilizer is $\mu_{4}$. So on the chart $a\neq 0$ we have a singularity of type $\frac{1}{2}(1,1)$. Note that $x=y=0$ implies $b = 0$. The singular point corresponds to a smooth elliptic curve of Weierstrass form $E_{4}$ and whose second marked point is $[0:0:1]$. If $x\neq 0$ then the stabilizer is $\mu_{2}$ and on this chart we find points of type $\frac{1}{2}(1,0)$ and these are smooth points. If $y\neq 0$ then $\xi\lambda^{3} = \lambda^3 = 1$ and we get a singularity of type $\frac{1}{3}(1,2)$ at the point $a=x=0$.\\
Consider now the locus $\{z=0\}$. We can take $y = 1$. Then $\xi\lambda^3 = 1$ and $X$ is given by $\{z=x^{3}+axz^{2}+bz^{3}\}$. We are interested in a neighborhood of $x = z = 0$. Let $f(x,z,a,b) = z-x^{3}-axz^{2}-bz^{3}$ be the polynomial defining $X$. Since $\frac{\partial f}{\partial z}_{|z=0} \neq 0$ we can choose $(x,a,b)$ as local coordinates. The action is given by $(\lambda,\xi,x,a,b)\mapsto (\xi\lambda^{2}x,\lambda^{4}a,\lambda^{6}b)$. If $x\neq 0$ the stabilizer is trivial. If $x = 0$ and $ab\neq 0$ the stabilizer is $\mu_{2}$ and does not produce any singularity. If $a = 0, b\neq 0$ then $\lambda^6 = 1$. This yields a singular point of type $\frac{1}{3}(1,1)$. If $a\neq 0, b = 0$ then $\lambda^4 = 1$ and we get a singular point of type $\frac{1}{2}(1,1)$ .
\end{proof}

\subsubsection{Weighted blow-ups}\label{genwbu}
The weighted blow-up of $\mathbb{A}^2$ with weights $(a_1,a_2)$ where $\gcd(a_1,a_2) = 1$ is a projective birational morphism $f:X\rightarrow\mathbb{A}^2$ such that $X$ is covered by two affine charts $U_1 \cong \mathbb{A}^2/\mu_{a_1}$ where the action is given by $\lambda\cdot(y_1,y_2)=(\lambda y_1,\lambda^{-a_2}y_2)$, and $U_2 \cong\mathbb{A}^2/\mu_{a_2}$ where $\lambda\cdot(y_1,y_2)=(\lambda^{-a_1}y_1,\lambda y_2)$. Note that the exceptional divisor $E$ is given in $U_i$ by  $\{y_i=0\}$. Therefore, $X$ has two finite quotient singularities of type 
\begin{equation}\label{singwbu}
\frac{1}{a_1}(1,-a_2) \quad \rm{and} \quad \frac{1}{\textit{a}_2}(-\textit{a}_1,1)
\end{equation}
on the exceptional divisor $E = \mathbb{P}(a_1,a_2)\cong \mathbb{P}^1$ corresponding to the points $[1:0]$, $[0:1]$ of $E$ respectively. By \cite[Section 3]{Hay} $X$ can be constructed as a follows. Let $C = \mathbb{R}_{+}e_1+\mathbb{R}_{+}e_2$ be the cone in $\mathbb{R}^2$ whose rays are $e_1 = (1,0)$, $e_2 = (0,1)$. We divide the cone $C$, in two cones $C_1$ and $C_2$, by adding the $1$-dimensional cone $\mathbb{R}_{+}e_3$, where $e_3 = (a_1,a_2)$. Let $\Sigma_{a_1,a_2}$ be the fan consisting of all the faces of $C_1, C_2$. Then $X$ is the toric variety determined by the lattice $N = \mathbb{Z}e_1+\mathbb{Z}e_2+\mathbb{Z}e_3$ and the fan $\Sigma_{a_1,a_2}$.

\begin{Theorem}\label{wbu}
The moduli space $\overline{M}_{1,2}$ is isomorphic to the weighted blow-up with weights $(2,3)$ of the weighted projective plane $\mathbb{P}(1,2,3)$ in its smooth point $[1:0:0]$. That is $\overline{M}_{1,2}$ is the toric variety associated to toric fan $\Sigma$ in the following picture: 
$$
\includegraphics[scale=0.45]{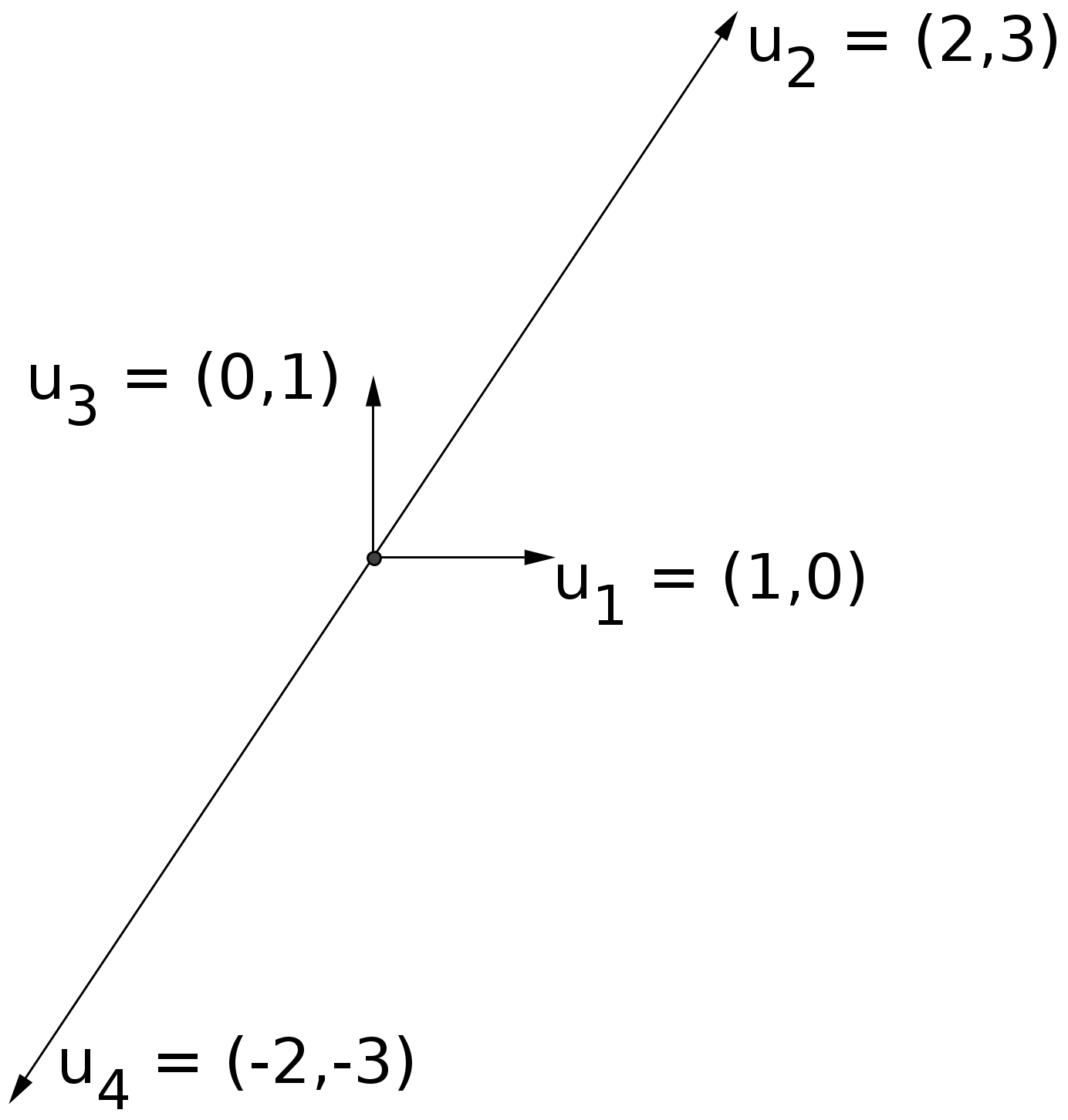} 
$$
\end{Theorem}   
\begin{proof}
Recall the description of $\overline{M}_{1,2}$ given via (\ref{act12}). On the chart $\mathcal{U}_{z} := \{z\neq 0\}$ we define a morphism 
\begin{equation}\label{fuz}
\begin{array}{cccc}
f_{\mathcal{U}_{z}}: & \mathcal{U}_{z} & \longrightarrow & \mathbb{P}(1,2,3)\\
 & (x,y,z,a,b) & \longmapsto & (x,az^{2},bz^{3})
\end{array}
\end{equation}
The action of $K^{*}\times K^{*}$ on this triple is given by $(\xi\lambda^{2},\xi^{2}\lambda^{4},\xi^{3}\lambda^{6})$, and $f_{\mathcal{U}_{z}}$ is indeed a well defined morphism to $\mathbb{P}(1,2,3)$. Note that the morphism $f_{\mathcal{U}_{z}}$ maps the two singular point in $M_{1,2}$ we found in Proposition \ref{M12} in the points $[0:1:0],[0:0:1]\in\mathbb{P}(1,2,3)$, which are the only singularities of the weighted projective plane and of the same type of the singularities on $M_{1,2}$.\\
On $\mathcal{U}_{y}:=\{y \neq 0\}$ the equation of $\overline{M}_{1,2}$ is $z=x^{3}+axz^{2}+bz^{3}$. So, as explained in the proof of Proposition \ref{M12} $x$ is a local parameter near $z=0$. We can consider the morphism 
$$f_{\mathcal{U}_{y}}(x,y,z,a,b) = \left( 1,a\left(\frac{x^{2}+az^{2}}{1-bz^{2}}\right)^{2}, b\left(\frac{x^{2}+az^{2}}{1-bz^{2}}\right)^{3}\right).$$
From this formulation it is clear that $f_{\mathcal{U}_{y}}$ is defined even on the locus $\{x = 0\}$ and the divisor $\Delta_{1} = \{x=z=0\}$ is contracted in the smooth point $[1:0:0]$ of $\mathbb{P}(1,2,3)$. On $\mathcal{U}_{z}\cap\mathcal{U}_{y}$ we have $\frac{z}{x} = \frac{x^{2}+az^{2}}{1-bz^{2}}$ and $f_{\mathcal{U}_{z}} = f_{\mathcal{U}_{y}}$, so $f_{\mathcal{U}_{z}}, f_{\mathcal{U}_{y}}$ glue to a morphism 
$$f:\overline{M}_{1,2}\rightarrow \mathbb{P}(1,2,3).$$
Since, by \cite{Mor} the Picard number of $\overline{M}_{1,2}$ is two, $\Delta_1$ must be the only divisor contracted by $f$. 
Let $\xi\in \mathbb{P}(1,2,3)$ be a point $\xi\neq [1:0:0]$. Since by (\ref{weier}) $z = 0$ forces $x = 0$ we have that $f^{-1}(\xi)$ is contained in the subset $\mathcal{U}_z$ where $z\neq 0$, and we may consider the expression for $f_{\mathcal{U}_z}$ in (\ref{fuz}). By (\ref{weier}) a point in the fiber $f_{\mathcal{U}_z}^{-1}(\xi)$ is determined, up to a sign, by the variable $y$. Therefore, in the quotient (\ref{act12}) such a point in uniquely determined. This means that $f_{|\overline{M}_{1,2}\setminus\Delta_1}:\overline{M}_{1,2}\setminus\Delta_1\rightarrow\mathbb{P}(1,2,3)\setminus\{[1:0:0]\}$ is an isomorphism. We conclude that $f:\overline{M}_{1,2}\rightarrow\mathbb{P}(1,2,3)$ is indeed a birational morphism with exceptional locus $\Delta_1$. Then, by \cite[Proposition 6.2.6]{Pr} $f$ is a weighted blow-up of $\mathbb{P}(1,2,3)$ in $[1:0:0]$ and $\Delta_{1}$ is the corresponding exceptional divisor. By Proposition \ref{M12} we know that $\overline{M}_{1,2}$ has two singularities of type $\frac{1}{2}(1,1)$, $\frac{1}{3}(1,1)$ on $\Delta_1$. By Proposition \ref{M12} we know that the singularities $\frac{1}{2}(1,1),\frac{1}{3}(1,1)\in\Delta_1$ corresponds to $b=0$, $a=0$ respectively. By (\ref{singwbu}) this forces the weights of the blow-up to be $(2,3)$.
\end{proof}

With this description of $\overline{M}_{1,2}$ we can easily compute its automorphism group.

\begin{Proposition}\label{autm12}
Let $K$ be an algebraically closed field with $\ch(K)\neq 2,3$. The automorphism group of $\overline{M}_{1,2}$ is isomorphic to $(K^{*})^{2}$.
\end{Proposition} 
\begin{proof}
Let $\pi:\overline{M}_{1,2}\rightarrow \overline{M}_{1,1}$ be a forgetful morphism, and $F_i = \pi^{-1}([E_i])$ for $i = 4,6$. Let $f:\overline{M}_{1,2}\rightarrow\mathbb{P}(1,2,3) = \Proj(K[x_0,x_1,x_2])$ be the weighted blow-up described in Theorem \ref{wbu}, where the grading on $K[x_0,x_1,x_2]$ is given by $\deg(x_i) = i+1$. Any $\phi\in\Aut(\overline{M}_{1,2})$ must stabilize $\Delta_1$, $F_4$ and $F_6$.\\ 
Therefore, we have an injective morphism $\chi:\Aut(\overline{M}_{1,2})\rightarrow G\subset\Aut(\mathbb{P}(1,2,3))$ where $G$ consists of automorphisms of $\mathbb{P}(1,2,3)$ fixing $[1:0:0]$ and stabilizing $f(F_4) = \{x_2 = 0\}$ and $f(F_6) = \{x_1 = 0\}$. To conclude that $G\cong (K^*)^{2}$ it is enough to recall that the automorphisms of $\mathbb{P}(1,2,3)$ are the automorphisms of the graded $K$-algebra $K[x_{0},x_1,x_{2}]$. By Theorem \ref{wbu} $\overline{M}_{1,2}$ is toric. To conclude it is enough to observe that the inclusion $i:(K^*)^2\hookrightarrow\Aut(\overline{M}_{1,2})$ is a section of $\chi$.
\end{proof}

For the rest of this section $K$ will be an algebraically closed field of characteristic zero. The cohomology of a toric invariant divisor $D$ in a toric variety $X_{\Sigma}$ can be computed in terms of combinatorial data encoded in the dual lattice of the variety, and in a polytope $P_{D}$ associated to $D$, see \cite[Chapters 4, 9]{CLS}. This leads us to the following result. 
\begin{Proposition}\label{cwbu}
Let $X_{\Sigma}$ we the weighted blow-up with weights $(2,3)$ of $\mathbb{P}(1,2,3)$ in $[1:0:0]$. Then $h^{0}(X_{\Sigma},T_{X_{\Sigma}})=2$ and $h^{i}(X_{\Sigma},T_{X_{\Sigma}})=0$ for $i =1,2$.
\end{Proposition}
\begin{proof}
Let $D_{\rho}$ be the toric divisor corresponding to the ray $u_{\rho}$ in the fan $\Sigma$ of Theorem \ref{wbu}, and let $\mathcal{O}_{X_{\Sigma}}(D_{\rho})$ be the associated sheaf. The cohomology of $\mathcal{O}_{X_{\Sigma}}(D_{\rho})$ can be computed in terms of lattice objects. Indeed, by \cite[Proposition 4.3.3]{CLS} we have $h^0(X_{\Sigma},\mathcal{O}_{X_{\Sigma}}(D_{\rho})) = 1$ for $\rho = 1,2,3,4$, while by \cite[Proposition 9.1.6]{CLS} we get $h^1(X_{\Sigma},\mathcal{O}_{X_{\Sigma}}(D_{\rho})) = h^2(X_{\Sigma},\mathcal{O}_{X_{\Sigma}}(D_{\rho})) = 0$ for $\rho = 1,2,3,4$.\\
Since the class group of $X_{\Sigma}$ if free of rank two, by \cite[Theorem 8.1.6]{CLS} we have the following exact sequence
$$0\mapsto \Omega_{X_{\Sigma}}\rightarrow \bigoplus_{\rho=1}^4\mathcal{O}_{X_{\Sigma}}(-D_{\rho})\rightarrow \mathcal{O}_{X_{\Sigma}}^{\oplus 2}\mapsto 0$$
and dualizing we get 
$$0\mapsto \mathcal{O}_{X_{\Sigma}}^{\oplus 2}\rightarrow \bigoplus_{\rho=1}^4\mathcal{O}_{X_{\Sigma}}(D_{\rho})\rightarrow T_{X_{\Sigma}}\mapsto 0$$
To conclude it is enough to take cohomology and to use the results on the cohomology of the toric divisors in the first part of the proof. 
\end{proof}

\begin{Remark}
By \cite[Proposition 6.4.4]{CLS} we can compute the intersection matrix of the $D_{\rho}$'s. This is given by
$$
(D_i\cdot D_j)=\left(\begin{array}{cccc}
0 & \frac{1}{3} & 0 & \frac{1}{3}\\ 
\frac{1}{3} & -\frac{1}{6} & \frac{1}{2} & 0\\ 
0 & \frac{1}{2} & 0 & \frac{1}{2} \\ 
\frac{1}{3} & 0 & \frac{1}{2} & \frac{1}{6}
\end{array}\right) 
$$
Therefore $D_{\rho}$ for $\rho = 1,3,4$ is nef, and the vanishing of $H^{i}(X_{\Sigma},\mathcal{O}_{X_{\Sigma}}(D_{\rho}))$ for $i = 1,2$ and $\rho = 1,3,4$ can be deduced also from Demazure's vanishing theorem \cite[Theorem 9.2.3]{CLS}.
\end{Remark}

\begin{Remark}
Note that by Proposition \ref{autm12} $\Aut(\overline{M}_{1,2})\cong (K^{*})^{2}$, and indeed by Proposition \ref{cwbu} we have that $T_{[Id]}\Aut(\overline{M}_{1,2})\cong H^0(\overline{M}_{1,2},T_{\overline{M}_{1,2}})$ is of dimension two.
\end{Remark}

Now, we are ready to compute the Kuranishi family of $\overline{M}_{1,2}$ by computing the deformations of the singularities of $\overline{M}_{1,2}$ described in Proposition \ref{M12}. See \cite{Ar} for a survey on deformations of singularities.

\begin{Theorem}\label{nonrig}
The coarse moduli space $\overline{M}_{1,2}$ does not have locally trivial deformations, while its family of first order infinitesimal deformations is non-singular of dimension six and the general deformation is smooth.
\end{Theorem}
\begin{proof}
The first order infinitesimal deformations of the surface $\overline{M}_{1,2}$ are parametrized by the group $\Ext^{1}(\Omega_{\overline{M}_{1,2}},\mathcal{O}_{\overline{M}_{1,2}})$. The sheaf $\mathcal{E}xt^{1}(\Omega_{\overline{M}_{1,2}},\mathcal{O}_{\overline{M}_{1,2}})$ is supported on the singularities and since $\overline{M}_{1,2}$ has isolated singularities $\mathcal{E}xt^{1}(\Omega_{\overline{M}_{1,2}},\mathcal{O}_{\overline{M}_{1,2}})$ can be computed separately for each singular point.\\
Let us consider the singular point $p\in \overline{M}_{1,2}$ of type $\frac{1}{3}(1,2)$. Then, \'etale locally, in a neighborhood of $p$ the surface $\overline{M}_{1,2}$ is isomorphic to $\mathbb{A}^{2}/\mu_{3}$ where the action is given by 
$$
\begin{array}{ccc}
\mu_{3}\times\mathbb{A}^{2} & \longrightarrow & \mathbb{A}^{2}\\
(\epsilon,x_{1},x_{2}) & \longmapsto & (\epsilon x_{1},\epsilon^{2}x_{2})
\end{array}
$$
The invariant polynomials with respect to this action are clearly $x_{1}^{3},x_{2}^{3},x_{1}x_{2}$. Therefore, \'etale locally, in a neighborhood of $p$ the surface $\overline{M}_{1,2}$ is isomorphic to an \'etale neighborhood of the singularity
$$S = \{f(x,y,z) = z^{3}-xy = 0\}\subset\mathbb{A}^{3}$$
Let $R = K[x,y,z]/(z^{3}-xy)$, and let us consider the free resolution
    \[
  \begin{tikzpicture}[xscale=1.5,yscale=-1.2]
    \node (A0_0) at (0, 0) {$0\mapsto R$};
    \node (A0_1) at (1, 0) {$R^{\oplus 3}$};
    \node (A0_2) at (2, 0) {$\Omega_R\mapsto 0$};
    \path (A0_0) edge [->]node [auto] {$\scriptstyle{\psi_J}$} (A0_1);
    \path (A0_1) edge [->]node [auto] {$\scriptstyle{}$} (A0_2);
  \end{tikzpicture}
  \]
of $\Omega_R$, where $\psi_J$ is the matrix of the partial derivatives of $f = z^3-xy$. Therefore, we get 
\begin{equation}\label{ext1}
\Ext^{1}(\Omega_{R},R)\cong R/\im(\psi_J^{t})\cong K[x,y,z]/(z^{3}-xy,-y,-x,3z^{2})\cong K[z]/(z^{2}), \: \Ext^{2}(\Omega_{R},R)=0
\end{equation}
The same argument for the two singularities of type $\frac{1}{2}(1,1)$ shows that in these cases we have
\begin{equation}\label{ext2}
\Ext^{1}(\Omega_{R},R)\cong K[x,y,z]/(z^{2}-xy,-y,-x,2z)\cong K, \: \Ext^{2}(\Omega_{R},R)=0
\end{equation}
Now, let us consider the singularity of type $\frac{1}{3}(1,1)$. In this case the action is given by 
$$
\begin{array}{ccc}
\mu_{3}\times\mathbb{A}^{2} & \longrightarrow & \mathbb{A}^{2}\\
(\epsilon,x_{1},x_{2}) & \longmapsto & (\epsilon x_{1},\epsilon x_{2})
\end{array}
$$
The invariants are clearly $x_{1}^{3},x_{1}^{2}x_2,x_{1}x_{2}^2,x_2^3$. Then \'etale locally, in a neighborhood of this singularity the surface $\overline{M}_{1,2}$ is isomorphic to an \'etale neighborhood of the vertex of a cone in $\mathbb{A}^4$ over a twisted cubic in $\mathbb{P}^3$. The coordinate ring of such a cone is $R = K[x,y,z,w]/(f_1,f_2,f_3)$, where $f_1 = xw-yz$, $f_2 = y^2-xz$ and $f_3 = z^2-yw$. We get a free resolution of the module of differentials $\Omega_R$ as follows
  \[
  \begin{tikzpicture}[xscale=1.5,yscale=-1.2]
    \node (A0_0) at (0, 0) {$R^{\oplus 3}$};
    \node (A0_1) at (1, 0) {$R^{\oplus 2}$};
    \node (A0_2) at (2, 0) {$R^{\oplus 3}$};
    \node (A0_3) at (3, 0) {$R^{\oplus 4}$};
    \node (A0_4) at (4, 0) {$\Omega_R\to 0$};
     \path (A0_3) edge [->]node [auto] {$\scriptstyle{}$} (A0_4);
    \path (A0_2) edge [->]node [auto] {$\scriptstyle{\psi_J}$} (A0_3);
    \path (A0_1) edge [->]node [auto] {$\scriptstyle{\psi_S}$} (A0_2);
    \path (A0_0) edge [->]node [auto] {$\scriptstyle{\psi}$} (A0_1);
    \end{tikzpicture}
  \]
where
$$
\psi = \left(\begin{array}{ccc}
x  & z & y\\ 
-y & -w & -z
\end{array}\right),
\quad
\psi_S = \left(\begin{array}{cc}
z & y \\ 
w & z \\ 
y & x
\end{array}\right),
\quad
\psi_J = \left(\begin{array}{ccc}
w & -z & 0 \\ 
-z & 2y & -w \\ 
-y & -x & 2z \\
x & 0 & -y
\end{array}\right)
$$
Note that $\psi_S$ and $\psi_J$ are the syzygy matrix and the Jacobian matrix of the $f_i$'s respectively. We may compute $\Ext^1(\Omega_R,R)\cong \ker(\psi_S^{t})/\im(\psi_J^{t})$. Now, $\ker(\psi_S^{t})$ has the following six generator: $(-w,z,0)$, $(z,-y,0)$, $(0,y,-w)$, $(y,0,-z)$, $(0,-x,z)$, $(-x,0,y)$, and $\im(\psi_J^{t})$ is generated by $(w,-z,0)$, $(-z,2y,-w)$, $(-y,-x,2z)$, $(x,0,-y)$. Furthermore, $\ker(\psi^{t})$ is generated by $(w,z)$, $(z,y)$, $(y,x)$, which are generators for $\im(\psi_S^{t})$ as well. Therefore 
\begin{equation}\label{ext3}
\dim_{K}\Ext^1(\Omega_R,R) = 2, \: \Ext^{2}(\Omega_{R},R)=0
\end{equation}
This last fact together with (\ref{ext1}) and (\ref{ext2}) implies that 
$$h^{0}(\overline{M}_{1,2},\mathcal{E}xt^{1}(\Omega_{\overline{M}_{1,2}},\mathcal{O}_{\overline{M}_{1,2}}))= 2+1+1+2 = 6$$ 
Now, by Theorem \ref{wbu} the surface $\overline{M}_{1,2}$ is a weighted blow-up of $\mathbb{P}(1,2,3)$ in its smooth point $[1:0:0]$ that is the variety $X_{\Sigma}$ in Proposition \ref{cwbu}. Therefore, Proposition \ref{cwbu} yields $h^0(\overline{M}_{1,2},T_{\overline{M}_{1,2}})=2$ and 
\begin{equation}\label{vanT}
H^{i}(\overline{M}_{1,2},T_{\overline{M}_{1,2}}) = 0 \quad \rm{for} \quad \textit{i}= 1,2
\end{equation}
So the sequence
$$H^{1}(\overline{M}_{1,2},T_{\overline{M}_{1,2}})\rightarrow\Ext^{1}(\Omega_{\overline{M}_{1,2}},\mathcal{O}_{\overline{M}_{1,2}})\rightarrow H^{0}(\overline{M}_{1,2},\mathcal{E}xt^{1}(\Omega_{\overline{M}_{1,2}},\mathcal{O}_{\overline{M}_{1,2}}))\rightarrow H^{2}(\overline{M}_{1,2},T_{\overline{M}_{1,2}})$$
yields $\Ext^{1}(\Omega_{\overline{M}_{1,2}},\mathcal{O}_{\overline{M}_{1,2}})\cong H^{0}(\overline{M}_{1,2},\mathcal{E}xt^{1}(\Omega_{\overline{M}_{1,2}},\mathcal{O}_{\overline{M}_{1,2}}))$.
Finally, to compute the dimension of the obstruction space $\Ext^{2}(\Omega_{\overline{M}_{1,2}},\mathcal{O}_{\overline{M}_{1,2}})$ we use the local-to-global Ext spectral sequence 
$$
H^i(\overline{M}_{1,2},\mathcal{E}xt^j(\Omega_{\overline{M}_{1,2}},\mathcal{O}_{\overline{M}_{1,2}}))\Rightarrow \Ext^{i+j}(\Omega_{\overline{M}_{1,2}},\mathcal{O}_{\overline{M}_{1,2}})
$$
Clearly $H^1(\overline{M}_{1,2},\mathcal{E}xt^1(\Omega_{\overline{M}_{1,2}},\mathcal{O}_{\overline{M}_{1,2}}))=0$ because $\mathcal{E}xt^1(\Omega_{\overline{M}_{1,2}},\mathcal{O}_{\overline{M}_{1,2}})$ is supported on a zero dimensional scheme. Furthermore, by (\ref{vanT}) we have $H^2(\overline{M}_{1,2},\mathcal{E}xt^0(\Omega_{\overline{M}_{1,2}},\mathcal{O}_{\overline{M}_{1,2}}))= H^{2}(\overline{M}_{1,2},T_{\overline{M}_{1,2}}) =0$. Finally (\ref{ext1}), (\ref{ext2}) and (\ref{ext3}) yield $H^0(\overline{M}_{1,2},\mathcal{E}xt^2(\Omega_{\overline{M}_{1,2}},\mathcal{O}_{\overline{M}_{1,2}}))= 0$ as well.
\end{proof}

\section{On the deformations of the coarse moduli space $\overline{M}_{g,n}$}\label{rcms}
In this section we will study the infinitesimal deformations of the coarse moduli space $\overline{M}_{g,n}$, which is a projective normal scheme with finite quotient singularities. If $g = 0$ we have $\overline{M}_{0,n}\cong\overline{\mathcal{M}}_{0,n}$, by \cite[Theorem 2.1]{Hac} in characteristic zero, and Theorem \ref{rig} in positive characteristic we know that $\overline{M}_{0,n}$ is rigid, hence we restrict to the case $g\geq 1$.

\subsection{Deformation theory for varieties with quotient singularities}
Let $X$ be a variety over a field $K$, $A$ an Artinian $K$-algebra with residue field $K$. A deformation $X_A$ of $X$ over $\Spec(A)$ is called {\em trivial} if it is isomorphic to $X\times_K\Spec(A)$; it is {\em locally trivial} if there is an open cover of $X$ by open affines $U$ such that the induced deformation $U_A$ is trivial.\\
We recall some well-known facts about infinitesimal deformations of normal varieties. By \cite{Ill} the tangent and obstruction spaces to deformations of $X$ are given by $\Ext^1(L_X,\mathcal{O}_X)$ and $\Ext^2(L_X,\mathcal O_X)$ where $L_X$ is the cotangent complex; when $X$ is a normal variety, these spaces are actually $\Ext^1(\Omega_X,\mathcal{O}_X)$ and $\Ext^2(\Omega_X,\mathcal{O}_X)$ respectively. Locally trivial infinitesimal deformations have as tangent and obstruction spaces $H^1(X,T_X)$ and $H^2(X,T_X)$, respectively. 

\begin{Remark} By the exact sequence 
$$0\mapsto H^{1}(X,T_{X})\rightarrow\Ext^{1}(\Omega_{X},\mathcal{O}_{X})\rightarrow H^{0}(X,\mathcal{E}xt^{1}(\Omega_{X},\mathcal{O}_{X}))\rightarrow H^{2}(X,T_{X})\to \Ext^2(\Omega_X,\mathcal O_X)$$
induced by the local-to-global spectral sequence for Ext, if $H^0(X,\mathcal{E}xt^1(\Omega_X,\mathcal O_X))=0$ then all deformations are locally trivial, while if $H^1(X,T_X)=0$ then all locally trivial deformations are trivial.
\end{Remark}

\begin{Remark}\label{highcodim} 
The sheaf $\mathcal{E}xt^1(\Omega_X,\mathcal O_X)$ is supported on the singular locus of $X$; if $X$ has quotient singularities, then by \cite[Lemmas 2.4, 2.5]{Fan} the sheaf of cohomology with supports $\mathcal H^0_Z(\mathcal{E}xt^1(\Omega_X,\mathcal O_X))=0$ if $Z\subset X$ is a closed subset of codimension greater or equal than three.
\end{Remark}

\subsection{Locally trivial deformations of $\overline{M}_{g,n}$}
Let $\mathcal{X}$ be a smooth Deligne-Mumford stack, and let $\pi:\mathcal{X}\rightarrow X$ be the structure morphism on its coarse module space. The stack $\mathcal{X}$ is called {\em canonical} if the locus where $\pi$ is not an isomorphism has dimension less or equal than $\dim(\mathcal{X})-2$ \cite[Definition 4.4]{FMN}. Every variety with quotient singularities is the coarse moduli space of a canonical smooth Deligne-Mumford stack, unique up to unique isomorphism \cite[Remark 4.9]{FMN}.

\begin{Lemma}\label{cohcs}
Let $X$ be a variety with finite quotient singularities and let $\mathcal{X}$ be its canonical stack. Then $H^i(X,T_X)\cong H^i(\mathcal{X},T_{\mathcal{X}})$ for any $i\geq 0$. In particular, if $\mathcal{X}$ is rigid then $X$ does not have locally trivial deformations.
\end{Lemma}
\begin{proof}
Let $Z = \Sing(X)$ and $U = X\setminus Z$. Since $\mathcal{X}$ is the canonical stack of $X$ we may embed $U$ in $\mathcal{X}$ as well. We get the following commutative diagram 
  \[
  \begin{tikzpicture}[xscale=1.5,yscale=-1.2]
    \node (A0_0) at (0, 0) {$U$};
    \node (A0_1) at (1, 0) {$\mathcal{X}$};
    \node (A1_1) at (1, 1) {$X$};
    \path (A0_0) edge [->]node [auto] {$\scriptstyle{i}$} (A0_1);
    \path (A0_1) edge [->]node [auto] {$\scriptstyle{\pi}$} (A1_1);
    \path (A0_0) edge [->,swap]node [auto] {$\scriptstyle{j}$} (A1_1);
  \end{tikzpicture}
  \]
and $\pi_{*}i_{*}T_{U} = j_{*}T_U$. Furthermore, we have the two following exact sequences:
$$
0\mapsto \mathcal{H}_{Z}^{0}(X,T_X)\rightarrow T_X\rightarrow j_{*}T_{X|U}\rightarrow \mathcal{H}^1_{Z}(X,T_X)$$
$$
0\mapsto \mathcal{H}_{W}^{0}(\mathcal{X},T_{\mathcal{X}})\rightarrow T_{\mathcal{X}}\rightarrow i_{*}T_{\mathcal{X}|U}\rightarrow \mathcal{H}^1_{W}(\mathcal{X},T_{\mathcal{X}})
$$
where $W = \pi^{-1}(Z)$ with the reduced substack structure, and $\mathcal{H}_{Z}^{i}$ is the sheaf of cohomology with supports, see \cite{Gro}. Furthermore, since $\codim_{X}(Z) = \codim_{\mathcal{X}}(W) \geq 2$ we get $\mathcal{H}_{Z}^{i}(X,T_X) = \mathcal{H}_{W}^{i}(\mathcal{X},T_{\mathcal{X}})=0$ for $i< 2$. This yields $\pi_{*}T_{\mathcal{X}}\cong \pi_{*}i_{*}T_U\cong j_{*}T_U\cong T_{X}$. Finally, since $R^{i}\pi_{*}T_{\mathcal{X}} = 0$ for $i\geq 1$, because $\pi$ is finite, we conclude by Leray's spectral sequence. 
\end{proof}

\begin{Lemma}\label{notquitecan} 
Let $c:\overline{\mathcal{M}}_{g,n}\rightarrow\overline{\mathcal{M}}_{g,n}^{can}$ be the structure map on the canonical stack. If $g+n\geq 4$, it is an isomorphism outside the boundary divisor $\Delta_1$ of elliptic tails, that is the image of $\overline{\mathcal{M}}_{1,1}\times \overline{\mathcal{M}}_{g-1,n+1}$ via the gluing map.
\end{Lemma}
\begin{proof}
By \cite[Corollary 1]{Co} if $g+n\geq 4$ any component of the locus parametrizing curves in $M_{g,n}$ with a non-trivial automorphism has codimension at least two.\\
Let us consider a general point $[C,x_1,...,x_n]$ of the boundary divisor $\Delta_{irr}$ parametrizing irreducible nodal curves. The normalization $\nu:\widetilde{C}\rightarrow C$ of $C$ is a smooth curve of genus $g-1$. Under our numerical hypothesis $[\widetilde{C},\nu^{-1}(x_1),...,\nu^{-1}(x_n)]$ is automorphism-free. Therefore, $[C,x_1,...,x_n]$ does not have non-trivial automorphisms as well.\\
The remaining boundary divisors parametrize reducible curves $[C_1\cup C_2,x_1,...,x_n]$ where $C_1$, $C_2$ are smooth curves of genus $g_1$, $g_2$ with $n_1$, $n_2$ marked points respectively intersecting just in one node $p = C_1\cap C_2$, and such that $g_1+g_2 = g$ and $n_1+n_2 = n$. Recalling that a general stable curve of the form $[C_i,x_{i_1},...,x_{i_{n_1}},p]$ has a non-trivial automorphism if and only if $g_i = 1$ and $n_i = 0$, we conclude that a general stable curve of the form $[C_1\cup C_2,x_1,...,x_n]$ admits a non-trivial automorphism if and only either $C_1$ or $C_2$ is an elliptic tail without marked points.
\end{proof}

The following statement should be compared to \cite[Theorem 2.3]{Hac}.
\begin{Theorem}\label{loctriv}
If $g+n\geq 4$, then the coarse moduli space $\overline{M}_{g,n}$ does not have non-trivial locally trivial deformations.
\end{Theorem}
\begin{proof}
Let $c:\overline{\mathcal{M}}_{g,n}\rightarrow\overline{\mathcal{M}}_{g,n}^{can}$ the structure map on the canonical stack. Let $f:Y = \overline{M}_{1,1}\times\overline{M}_{g-1,n+1}\rightarrow \overline{M}_{g,n}$ be the natural attaching morphism. Since $g+n\geq 4$, by Lemma \ref{notquitecan} we have that $f(Y)$ is the locus in $\overline{M}_{g,n}$ where $c:\overline{\mathcal{M}}_{g,n}\rightarrow\overline{\mathcal{M}}_{g,n}^{can}$ is not an isomorphism. Now, on $Y$ we have an exact sequence
$$0\mapsto T_Y\rightarrow f^{*}T_{\overline{\mathcal{M}}_{g,n}}\rightarrow L\mapsto 0$$
where $L\cong \psi_{1}^{\vee}\boxtimes\psi_{n+1}^{\vee}$. Since the general curve parametrized by $Y$ has two automorphisms the line bundle $L$ fits in the following exact sequence as well
$$0\mapsto c_{*}T_{\overline{\mathcal{M}}_{g,n}}\rightarrow T_{\overline{\mathcal{M}}_{g,n}^{can}}\rightarrow c_{*}f_{*}L^{\otimes 2}\mapsto 0$$
Since $L^{\vee}\cong \psi_{1}\boxtimes\psi_{n+1}$ is big and nef \cite[Theorem A.1]{Hac} yields 
$$H^{0}(\overline{\mathcal{M}}_{g,n}^{can},c_{*}f_{*}L^{\otimes 2}) = H^{1}(\overline{\mathcal{M}}_{g,n}^{can},c_{*}f_{*}L^{\otimes 2})=0$$
which in turns implies $H^{1}(\overline{\mathcal{M}}_{g,n}^{can},T_{\overline{\mathcal{M}}_{g,n}^{can}})\cong H^{1}(\overline{\mathcal{M}}_{g,n}^{can}, c_{*}T_{\overline{\mathcal{M}}_{g,n}})\cong H^{1}(\overline{\mathcal{M}}_{g,n},T_{\overline{\mathcal{M}}_{g,n}})$. To conclude it is enough to recall that the stack $\overline{\mathcal{M}}_{g,n}$ is rigid and to apply Lemma \ref{cohcs} with $\mathcal{X} = \overline{\mathcal{M}}_{g,n}^{can}$ and $X = \overline{M}_{g,n}$.
\end{proof} 

\subsection{Singularities of $\overline{M}_{g,n}$}
We denote again by $\Delta_{1}$ the image in $\overline{M}_{g,n}$ of $\overline{M}_{1,1}\times \overline{M}_{g-1,n+1}$, the divisor parametrizing curves with elliptic tails. We introduce the following notation for codimension two, that is of maximal dimension, irreducible components of the singular locus of $\overline{M}_{g,n}$:
\begin{itemize}
\item[-] $Z_i$ the image of $[E_i]\times \overline{M}_{g-1,n+1}\subset \Delta_{1}$ for $i = 4,6$, the codimension two loci where the elliptic tail has four and six automorphisms respectively;
\item[-] $Y$ the locus parametrizing reducible curves $E\cup C$ where $E$ is an elliptic curve with a marked point which is fixed by the elliptic involution, and $C$ is a curve of genus $g-1$ with $n-1$ marked points;
\item[-] $W$ the locus parametrizing reducible curves $C_1\cup C_2$ where $C_1$ and $C_2$ are of genus two and $g-2$ respectively, the marked points are on $C_2$, and $C_1\cap C_2$ is a fixed point of the hyperelliptic involution on $C_1$.
\end{itemize}

\begin{Definition}\label{transing}
A variety $X$ which, \'etale locally at a general point of its reduced singular locus $Z$, has type $\frac{1}{n}(a_1,a_2,0,\ldots,0)$ will be said to have a transversal $\frac{1}{n}(a_1,a_2)$ singularity along $Z$.
\end{Definition}

\begin{Proposition}\label{sing}
If $g+n > 4$, then the only codimension two irreducible components of $\Sing(\overline{M}_{g,n})$ are $Z_4, Z_6, Y$ and $W$. Each component contains dense open subsets, denoted by a superscript zero, with complement of codimension at least two such that $\overline{M}_{g,n}$ has transversal $A_1$ singularities along $Z^0_4, Y^0$ and $W^0$, and transversal $\frac{1}{3}(1,1)$ singularities along $Z^0_6$. 
\end{Proposition}
\begin{proof}
By \cite[Proposition 1]{Pa} under our numerical hypothesis any component of $\Sing(M_{g,n})$ has codimension at least three in $M_{g,n}$. By Lemma \ref{notquitecan} the divisor $\Delta_1$ is the only boundary divisor whose general point corresponds to a curve with non-trivial automorphisms. There are two ways of producing loci of codimension two in $\overline{M}_{g,n}$ parametrizing curves with extra automorphisms in $\Delta_{1}$. Namely,
\begin{itemize}
\item[(\textit{a})] we may add another elliptic tail,
\item[(\textit{b})] we consider the loci $Z_i = [E_i]\times \overline{M}_{g-1,n+1}$ for $i = 4,6$.
\end{itemize} 
Now, note that the automorphism group of the general curve parametrized by (\textit{a}) is generated by pseudo-reflections. Therefore, this locus is not singular for $\overline{M}_{g,n}$.\\ 
Any other component must be properly contained in a different boundary divisor, hence must be induced by a codimension one stratum with non-trivial automorphisms in some ${M}_{g_1,n_1+1}$ with $(g_1,n_1)\ne (1,0)$. There are exactly two such strata, for $(g_1,n_1+1)=(1,2)$ and $(2,1)$, yielding the components $Y$ and $W$.\\
The automorphism groups of a general curve parametrized by $Z_4,Z_6,Y$ and $W$ are not generated by pseudo-reflections. We conclude that the codimension two components of $\Sing(\overline{M}_{g,n})$ are exactly $Z_4,Z_6,Y$ and $W$.\\
Furthermore, there is just one way to produce a divisor in $Z_i$ parametrizing curves with extra automorphisms, namely adding another elliptic tail;  this however generates a pseudo-reflection. We conclude the locus in $Z_i$ where the singularities are different from the generic point has codimension at least two. The same argument applies to $Y$ and $W$.\\
The finite quotient singularities in $Y^0$ and $W^0$ are both produced by automorphisms of order two. Therefore, these singularities are forced to be of type $A_1$.\\
The type of the singularities along the $Z_i^{0}$'s is the same for any $g\geq 1$, $n\geq 2$. Therefore, by Lemma \ref{facile} and the second part of Proposition \ref{M12} we have that $Z_4^{0}$ is a singularity of type $\frac{1}{2}(1,1,0,...,0)$, or transversal $A_1$, and $Z_6^{0}$ is a singularity of type $\frac{1}{3}(1,1,0,...,0)$ or transversal $\frac{1}{3}(1,1)$.
\end{proof}

\subsection{Deformation of varieties with transversal $A_1$ singularities}

In this section we assume that the characteristic of the ground field is different from two. Let $\mathcal{X}$ be a smooth Deligne-Mumford stack, and assume that its inertia stack is the disjoint union of $\mathcal{X}$ and of a smooth stack $\mathcal{Z}$, mapping isomorphically to its image in $\mathcal{X}$, which we also denote by $\mathcal{Z}$. Let $\eps:\mathcal{X}\to X$ be the morphism to the coarse moduli space, which we assume is a scheme. Let $\bar \eps:\mathcal{Z}\to Z$ be the morphism to the coarse moduli space, and assume that $\bar \eps$ is a $\mu_2$-gerbe.\\ 
An example can be constructed as follows: let $Z$ be a non-singular variety, and $F$ a rank two vector bundle on $Z$. Let $\mu_2$ act on $F$ by scalar multiplication. Then $\mathcal{X}=[F/\mu_2]$ and $\mathcal{Z}=Z\times B\mu_2$ satisfy the assumptions.\\
In this case, $X$ is a variety with transversal $A_1$, or transversal $\frac{1}{2}(1,1)$ singularities along $Z$, that is \'etale locally, the inclusion $i:Z\to X$ looks like $Z=Z\times\{0\}\to Z\times \Spec K[x,y,z]/(xy-z^2)$. Conversely, if $X$ has transversal $A_1$ singularities along $Z$, then such an $\eps:\mathcal{X}\to X$ exists, it is the canonical stack, and is unique up to unique isomorphism, while $\mathcal{Z}=\eps^{-1}(Z)_{red}$.\\
Let $N:=N_{\mathcal{Z}/\mathcal{X}}$ be the normal bundle of $\mathcal{Z}$ in $\mathcal{X}$; there exist a, unique up to isomorphism, rank three vector bundle $V$ on $Z$ and a, unique up to isomorphism, line bundle $L$ on $Z$, such that $\bar \eps^*(E)=\Sym^2N$ and $\bar \eps^*(L)=\det N$.
Let $\pi:V\to Z$ be the projection, and let  $\bar L:=\pi^*L$. Let $C$ be the coarse moduli space of $N$; $C$ has transversal $A_1$ singularities along $Z$, it is easy to see that $C=C_{Z/X}$, the normal cone of $Z$ in $X$, but we will not need this.

\begin{Lemma}\label{bar1}
There is a natural embedding of $C$ in $V$ as the scheme-theoretic zero section of $\bar L^{\otimes 2}$.
\end{Lemma}
\begin{proof} Let us consider the principal $GL(2)$ bundle $P\to\mathcal{Z}$  associated to $N$, and denote by $N_P$ the pull-back of $N$ to $P$, which is canonically trivial as each point of $P$ corresponds to a basis of a fiber of $N$. A basis $n_1$, $n_2$ of $N_{z}$ induces a basis $n_{11}, n_{12}, n_{22}$ of $E_P$ hence a trivialization, with coordinates $x_{11}, x_{12}, x_{22}$, $x_{ij}=n_in_j$.  The rank one tensors are given by the equation $f=x_{11}x_{22}-x_{12}^2$. It is easy to see that the section $f(n_1\wedge n_2)^{\otimes 2}$ of $L_P$ is $GL(2)$-invariant.
\end{proof}

\begin{Theorem}\label{211}
In our assumptions we have $\Exts^1(\Omega_X,\mathcal{O}_X)=i_*\bar L^{\otimes 2}$, where $i:Z\to X$ is the natural embedding.
\end{Theorem}
\begin{proof} Assume first that $\mathcal{X}=N$, hence $X=C$; then the result immediately follows from Lemma \ref{bar1}, by applying the functor $\mathcal{H}om(-, \mathcal{O}_X)$ to the exact sequence 
$$0\mapsto \mathcal{O}_{V}(-C)_{|C}\to \Omega_{V|C}\to \Omega_C\mapsto 0$$
For the general case, we first prove that the sheaf is the push-forward of a line bundle on $Z$, since this statement is \'etale local hence we reduce to the previous case. To identify the line bundle, we use the degeneration to the normal cone of the inclusion $\mathcal{Z}\to \mathcal{X}$; we get a one-parameter family of line bundles on $Z$ where the general one is $\Exts^1(\Omega_X,\mathcal{O}_X)$ and the special one is $\bar{L}^{\otimes 2}$; we conclude that they are equal since the Picard scheme of $Z$ is separated.
\end{proof}

\subsection{Deformation of varieties with transversal $\frac{1}{3}(1,1)$ singularity}

In this section we assume the characteristic to be different from two and three. Again we start with a special case. Assume that $Z$ is a smooth variety, and $F$ the total space of a rank $2$ vector bundle on $Z$. Let the group $\mu_3$ act on $F$ by scalars, and let $\mathcal{X}=[F/\mu_3]$ and $\eps:\mathcal{X}\to X$ be the morphism to the coarse moduli space, which is an affine variety, and indeed a cone over $Z$.\\
Assume now that $F=L_1\oplus L_2$ with $L_i$ line bundles on $Z$. Write $L_{ij}:=L_1^{\otimes i}\otimes L_2^{\otimes j}\in \Pic (Z)$. Let $V=\Sym^3F=L_{03}\oplus L_{12}\oplus L_{21}\oplus L_{30}$. Let $\pi:V\to Z$ be the projection, and write $\bar L_{ij}:=\pi^*(L_{ij})$. Finally, let $\sigma$ be the tautological section of $\pi^*V$, and write $\sigma=(s_0,s_1,s_2,s_3)$ with $s_i$ a section of $\bar L_{i,3-i}$.\\
There is a natural embedding $X\to V$, where $X$ parametrizes rank one tensors in $\Sym^3(F)$. Write $\widetilde L_{ij}:=\bar L_{ij|X}$.

\begin{Lemma}\label{bar2}
There is an exact sequence of coherent sheaves on $X$ 
$$\widetilde L_{54}^\vee\oplus \widetilde L_{45}^\vee\to \widetilde L_{24}^\vee\oplus \widetilde L_{33}^\vee \oplus \widetilde L_{42}^\vee\to \widetilde L_{03}^{\vee}\oplus \widetilde L_{12}^\vee \oplus \widetilde L_{21}^\vee \oplus \widetilde L_{03}^\vee\oplus \pi^*\Omega_{Z|X}\to \Omega_{X}\mapsto 0$$
where the first map $\alpha$ is given by the matrix $$
\left(\begin{array}{cc}
s_3 & s_2\\ 
s_2 & s_1\\ 
s_1 & s_0
\end{array}\right)$$
and the second map $\beta$ is given by the matrix $$
\left(\begin{array}{ccc}
s_2 & -s_3 & 0\\
-2s_1 & s_2 & s_3\\
s_0 & s_1 & -2s_2 \\
0 & -s_1 & s_0\\
\end{array}\right)$$
\end{Lemma}
\begin{proof} Note that $X$ is the scheme-theoretic zero locus in $V$ of the section $(s_0s_2-s_1^2, s_1s_2-s_0s_3, s_1s_3-s_2^2)$ of $\bar L_{24}\oplus \bar L_{33}\oplus \bar L_{42}$. The rest of the statement is elementary homological algebra.
\end{proof}

\begin{Proposition}\label{bar3} 
The sheaf $\Exts^1(\Omega_X,{\mathcal O}_X)$ is isomorphic to $i_*(L_{21}\oplus L_{12})=i_*(F)\otimes\det F$.
\end{Proposition}
\begin{proof} Let $\alpha$ be the map in Lemma \ref{bar2}. Then $\alpha^\vee$ is the second map in an exact sequence $$
\widetilde L_{21}\oplus \widetilde L_{12}\oplus \widetilde L_{03}\oplus \widetilde L_{30}\oplus \widetilde L_{21}\oplus \widetilde L_{11} \to \widetilde L_{24}\oplus \widetilde L_{33}\oplus \widetilde L_{42}\to \widetilde L_{54}\oplus\widetilde L_{45},$$
where the first map $\gamma$ is given by the matrix 
$$\left(\begin{array}{cccccc}
s_0 & s_1 & s_2 & 0 & 0 & 0\\
-s_1 & -s_2 & -s_3 &s_0 & s_1 & s_2 \\
 0 & 0 & 0 & -s_1 & -s_2 & -s_3\\
\end{array}\right)$$
We conclude by observing that ${\mathcal I}_{Z/X}\cdot \Exts^1(\Omega_X, {\mathcal O}_X)=0$ and by comparing the two exact sequences.
\end{proof}

Recall that by Definition \ref{transing} a variety $X$ which, \'etale locally near each point of its reduced singular locus $Z$, has type $\frac{1}{3}(1,1,0,\ldots,0)$ will be said to have a transversal $\frac{1}{3}(1,1)$ singularity along $Z$. We denote by $\eps:\mathcal{X}\to X$ the associated canonical stack. Let $\mathcal{Z}:=\eps^{-1}(Z)_{red}$ and $\bar\eps:=\eps|_{\mathcal{Z}}:\mathcal{Z}\to Z$, which is a $B\mu_3$-gerbe. 

\begin{Corollary}\label{311}
Assume that $X$ has a transversal $\frac{1}{3}(1,1)$ singularity along $Z$, and moreover that there exist line bundles $L_1, L_2$ on $\mathcal Z$ such that $N:=N_{\mathcal{Z}/\mathcal{X}}$ fits into an exact sequence 
$$
0\mapsto \bar L_1\to N\to \bar L_2\mapsto 0
$$
Write $ \bar L_{ij}:=\bar L_1^{\otimes i}\otimes \bar L_2^{\otimes j}$; when $i+j$ is a multiple of three, write $L_{ij}$ for the line bundle on $Z$ whose pull-back to $\mathcal Z$ is $\bar L_{ij}$. If $H^0(Z,L_{12})=H^0(Z,L_{21})=0$ then we have $H^0(X,\Exts^1(\Omega_X,{\mathcal O}_X))=0$.
\end{Corollary}
\begin{proof}
Since \'etale locally we are in the situation of the Proposition \ref{bar3}, we have $\Exts^1(\Omega_X,{\mathcal O}_X)=i_*E$ where $E$ is a rank two vector bundle on $Z$ and $i:Z\to X$ is the inclusion. We can again degenerate the inclusion $\mathcal{Z}\to \mathcal{X}$ to the normal cone, this degenerates $E$ to some other line bundle $E'$, and by semi-continuity it will be enough to show that $E'$ does not have global sections on $Z$. We are now assuming that $\mathcal{X}=N$; again we can degenerate $N$ to $\bar\eps^*L_1\oplus\bar\eps^*L_2$, hence we are now in the set-up of Proposition \ref{bar3}, and $E'$ degenerates to $L_{12}\oplus L_{21}$. By our assumptions the latter has no non-zero global sections.
\end{proof}

We are ready to prove the main result of this section.

\begin{Theorem}\label{rigcms}
If $g+n> 4$, over an algebraically closed field of characteristic zero, the coarse moduli space $\overline{M}_{g,n}$ is rigid.
\end{Theorem}
\begin{proof}
By Theorem \ref{loctriv} we know that $\overline{M}_{g,n}$ does not have locally trivial deformations. To conclude it is enough to prove that $H^{0}(\overline{M}_{g,n},\mathcal{E}xt^{1}(\Omega_{\overline{M}_{g,n}},\mathcal{O}_{\overline{M}_{g,n}}))=0$.\\
Now, $\mathcal{E}xt^{1}(\Omega_{\overline{M}_{g,n}},\mathcal{O}_{\overline{M}_{g,n}})$ is a coherent sheaf supported on $\Sing(\overline{M}_{g,n})$. Let $S_3\subset \Sing(\overline M_{g,n})$ be the union of all irreducible components which have codimension greater or equal to three. By Remark \ref{highcodim} there are no sections of $\mathcal{E}xt^{1}(\Omega_{\overline{M}_{g,n}},\mathcal{O}_{\overline{M}_{g,n}})$ supported on $S_3$. By Proposition \ref{sing} we know the codimension two components of $\overline{M}_{g,n}$ are $Z_6$, $Z_4$, $Y$, and $W$, and again there are no sections supported on $Z_6\setminus Z_6^0$, since by Proposition \ref{sing} $Z_6^0$ has codimension at least two in $Z_6$, and similarly for $Z_4$, $Y$, and $W$. Thus it will be enough to show $H^0(Z_6^0, \mathcal{E}xt^{1}(\Omega_{\overline{M}_{g,n}},\mathcal{O}_{\overline{M}_{g,n}})_{|Z_6^0})=0$ and similarly for $Z_4^0$, $Y^0$ and $W^0$. We will give in detail the proof for $Z_6^0$ and sketch the other cases.\\ 
Let $\mathcal Z_6$ be the image in $\overline{\mathcal M}_{g,n}$ of $[E_6]\times \overline{\mathcal M}_{g-1,n+1}=B\mu_6\times \overline{\mathcal M}_{g-1,n+1}$; let $\mathcal Z_6^{can}$ be its image in $\overline{\mathcal M}^{can}_{g,n}$; $\mathcal Z_6^{can}$ is the inverse image of $Z_6$ in the canonical stack, with the reduced induced structure. Denote as usual by $\Delta_1\subset \overline{\mathcal M}_{g,n}$ the divisor of rational tails, and let $\Delta_1^{can}$ be its image in $\overline{\mathcal M}^{can}_{g,n}$. Consider the commutative diagram
  \[
  \begin{tikzpicture}[xscale=3.5,yscale=-1.2]
    \node (A0_0) at (0, 0) {$B\mu_6\times\overline{\mathcal{M}}_{g-1,n+1}$};
    \node (A0_1) at (1, 0) {$\overline{\mathcal{M}}_{1,1}\times\overline{\mathcal{M}}_{g-1,n+1}$};
    \node (A1_0) at (0, 1) {$\mathcal{Z}_6$};
    \node (A1_1) at (1, 1) {$\Delta_1$};
    \node (A1_2) at (2, 1) {$\overline{\mathcal{M}}_{g,n}$};
    \node (A2_0) at (0, 2) {$\mathcal{Z}_6^{can}$};
    \node (A2_1) at (1, 2) {$\Delta_1^{can}$};
    \node (A2_2) at (2, 2) {$\overline{\mathcal{M}}_{g,n}^{can}$};
    \path (A0_1) edge [->]node [auto] {$\scriptstyle{p_1}$} (A1_1);
    \path (A0_0) edge [->]node [auto] {$\scriptstyle{i}$} (A0_1);
    \path (A2_0) edge [->]node [auto] {$\scriptstyle{i_2}$} (A2_1);
    \path (A1_0) edge [->]node [auto] {$\scriptstyle{i_1}$} (A1_1);
    \path (A1_1) edge [->]node [auto] {$\scriptstyle{j_1}$} (A1_2);
    \path (A1_0) edge [->,swap]node [auto] {$\scriptstyle{f_2}$} (A2_0);
    \path (A1_1) edge [->]node [auto] {$\scriptstyle{p_2}$} (A2_1);
    \path (A1_2) edge [->]node [auto] {$\scriptstyle{q}$} (A2_2);
    \path (A0_0) edge [->,swap]node [auto] {$\scriptstyle{f_1}$} (A1_0);
    \path (A0_0) edge [->,bend left=20,swap]node [auto] {$\scriptstyle{f}$} (A2_0);
    \path (A2_1) edge [->]node [auto] {$\scriptstyle{j_2}$} (A2_2);
    \path (A1_2) edge [->]node [auto] {$\scriptstyle{}$} (A2_2);
  \end{tikzpicture}
  \]
where the horizontal arrows are closed embeddings, and the vertical ones are proper, finite morphisms. Let $U$ be the inverse image in $B\mu_6\times \overline{\mathcal M}_{g-1,n+1}$ of $Z_6^0$; at all points of $U$, and their images, all the stacks of the diagram are smooth and all vertical maps are bijective on points; we have that $N_{i|U}=f^*N_{i_2|U}$ while $f^*i_2^*N_{j_2|U}=f_1^*i_1^*N_{j_1|U}^{\otimes 2}$ since the morphism $p_1$ is ramified along $\Delta_1$ of order two. Now we have that $N_i=pr_1^*T_{\overline{M}_{1,1}|B\mu_6}$ while $p_1^*N_{j_1}=pr_1^*\psi_1^\vee\otimes pr_2^*\psi_{n+1}^\vee$. To prove that $H^0(Z_6^0, \mathcal{E}xt^{1}(\Omega_{\overline{M}_{g,n}},\mathcal{O}_{\overline{M}_{g,n}})_{|Z_6^0}))=0$ is equivalent, in view of Corollary \ref{311} together with the fact that $f$ over $Z_6^0$ is finite and bijective on points, to show that $$
H^0(U, N_i^{\otimes 2}\otimes (pr_1^*\psi_1^\vee\otimes pr_2^*\psi_{n+1}^\vee)^{\otimes 2}_{|U})=
H^0(U, N_i\otimes (pr_1^*\psi_1^\vee\otimes pr_2^*\psi_{n+1}^\vee)^{\otimes 4}_{|U})=0$$
Since the complement of $U$ in $B\mu_6\times \overline{\mathcal M}_{g-1,n+1}$ has codimension two, it is enough to prove that $$
H^0(B\mu_6\times \overline{\mathcal M}_{g-1,n+1}, N_i^{\otimes 2}\otimes (pr_1^*\psi_1^\vee\otimes pr_2^*\psi_{n+1}^\vee)^{\otimes 2})=
H^0(B\mu_6\times \overline{\mathcal M}_{g-1,n+1}, N_i\otimes (pr_1^*\psi_1^\vee\otimes pr_2^*\psi_{n+1}^\vee)^{\otimes 4})$$
are zero. By \cite[Sections 2, 3]{Hac} psi-classes are nef and big. To conclude it is enough to use Kodaira vanishing \cite[Theorem A.1]{Hac}.\\
For the components $Z_4^0,Y^0,W^0$ the relevant commuting diagram is obtained in the same way (by taking images in $\overline{\mathcal M}_{g,n}$ in the second row and in $\overline{\mathcal M}^{can}_{g,n}$ in the third row), but taking as first row the following embeddings:
\begin{itemize}
\item[-] for $Z_4^0$, $B\mu_4\times \overline{\mathcal M}_{g-1,n+1}\to \overline{\mathcal M}_{1,1}\times \overline{\mathcal M}_{g-1,n+1}$ where $B\mu_4$ maps to the point $[E_4]$;
\item[-] for $Y^0$, $A\times \overline{\mathcal M}_{g-1,n}\to \overline{\mathcal M}_{1,2}\times \overline{\mathcal M}_{g-1,n+2}$ where $A\subset \overline{\mathcal M}_{1,2}$ is the closure of the locus parametrizing triples $(E,p_1,p_2)$ where $(E,p_1)$ is a smooth elliptic curve and $p_2$ is a $2-$torsion point in $E$;
\item[-] for $W^0$, $A'\times \overline{\mathcal M}_{g-2,n+1}\to \overline{\mathcal M}_{2,1}\times \overline{\mathcal M}_{g-2,n+1}$ where $A'\subset \overline{\mathcal M}_{2,1}$ is the closure of the locus parametrizing triples $(C,p_1)$ where $C$ is a smooth genus $2$ curve and $p_1$ is apoint fixed by the hyperellitpic involution.
\end{itemize}
We argue in an analogous way and use Theorem \ref{211} instead of Corollary \ref{311}, and conclude as before using the fact that $\psi$ classes are nef and big.
\end{proof}

\section{Automorphisms of $\overline{M}_{g,n}$ over an arbitrary field}\label{aaf}

In this section we will use the main results of Sections \ref{autrig} and \ref{g0} to compute the automorphism groups of $\overline{M}_{g,n}$ over an arbitrary field. In order to do this we will apply the theory developed in Section \ref{autrig} taking $A = W(K)$, that is the ring of Witt vectors over $K$, see \cite{Wi} for details.\\ 
For our purposes it is enough to keep in mind that $W(K)$ is a discrete valuation ring with a closed point $x\in\Spec(W(K))$ with residue field $K$, and a generic point $\xi\in\Spec(W(K))$ with residue field of characteristic zero.

\begin{Proposition}\label{rigidautoscheme} Assume that $\Aut(\overline{M}^{\overline K}_{g,n})\cong S_n$ where $K$ is any field of characteristic zero. Then $\Aut(\overline{M}^{K}_{g,n})\cong S_n$. Furthermore, if $H^{0}(\overline{M}^{K}_{g,n},T_{\overline{M}^{K}_{g,n}})=0$ and $\overline{M}^{K_p}_{g,n}$ is rigid for some field $K_p$ of characteristic $p$ then $\Aut(\overline{M}^{K_p}_{g,n})\cong S_n$.
\end{Proposition}
\begin{proof} 
By Lemma \ref{c1} there is an injective morphism of groups
$$\chi:\Aut(\overline M^{K}_{g,n})\rightarrow \Aut(\overline M^{\overline K}_{g,n})$$
The hypothesis $\Aut(\overline M^{\overline K}_{g,n})\cong S_n$ forces $\chi$ to be an isomorphism. Let $A$ be a ring of Witt vectors $W(K_p)$ of $K_p$ and $K$ its quotient field. Recall that $\ch(K)=0$. By Theorem \ref{lift} we have an injective morphism of groups
$$\chi_p:\Aut(\overline{M}^{K_p}_{g,n})\rightarrow \Aut(\overline M^{K}_{g,n})$$
By the first part of the proof $\Aut(\overline M^{K}_{g,n})\cong S_n$, and $\chi_p$ is an isomorphism.
\end{proof}

Finally, we can prove analogue of Theorem \ref{rigidautoscheme} for the stack $\overline{\mathcal{M}}_{g,n}$.

\begin{Proposition}\label{rigidautostack} Assume that $\Aut(\overline{M}^{\overline K}_{g,n})\cong S_n$ and $H^{0}(\overline{\mathcal{M}}^{K}_{g,n},T_{\overline{\mathcal{M}}^{K}_{g,n}})=0$ for $K$ any field of characteristic zero. If for some field $K_p$ of characteristic $p$ the stack $\overline{\mathcal{M}}^{K_p}_{g,n}$ is rigid then $\Aut(\overline{\mathcal{M}}^{K_p}_{g,n})\cong S_n$.
\end{Proposition}
\begin{proof} Let again $A$ be a ring of Witt vectors $W(K)$ of $K_p$ with residue field $K$. Let $\phi\in \Aut(\overline {\mathcal M}^{K_p}_{g,n})$ be an automorphism. As in the proof of Theorem \ref{rigidautoscheme}, the rigidity assumption and semi-continuity imply that $H^0(\overline{\mathcal{M}}_{g,n}^{K_p},T_{\overline{\mathcal{M}}_{g,n}^{K_p}})=0$ hence by Lemma \ref{liftaut} $\phi$ is the restriction of a unique $\phi_i\in \Aut(\overline{\mathcal{M}}_{g,n}^{A_i}/A_i)$ for every $i\ge 1$. By Proposition \ref{autstack} the automorphism $\phi_i$ induces an automorphism $\tilde{\phi}_i\in \Aut (\overline{M}_{g,n}^{A_i}/A_i)$ such that the restriction of $\tilde{\phi}_{i+1}$ to $\overline{M}_{g,n}^{A_i}$ is $\tilde{\phi}_i$. Thus, reasoning as in Theorem \ref{lift} there exists a unique automorphism $\tilde{\phi}_A\in \Aut(\overline M_{g,n}^A/A)$ inducing $\tilde{\phi}_i$ for any $i\geq 0$. Let $\tilde{\phi}_{|\overline M_{g,n}^K}$ be the restriction of $\tilde{\phi}$ to $\overline M_{g,n}^K$. This construction yields and injective morphism of groups
$$
\begin{array}{cccc}
\chi: & \Aut(\overline{\mathcal{M}}_{g,n}^{K_p}) & \longrightarrow & \Aut(\overline{M}_{g,n}^{K})\\
 & \phi & \longmapsto & \tilde{\phi}_{|\overline{M}_{g,n}^{K}}
\end{array}
$$
Now, to conclude it is enough to argue as in the proof of Proposition \ref{rigidautoscheme}.
\end{proof}

Now, we are ready to prove the main result of this section.

\begin{Theorem}\label{aut1}
Let $K$ be any field then 
$$\Aut(\overline{M}_{0,n})\cong S_{n}$$
for any $n\geq 5$. Furthermore, if $K$ is any field with $\ch(K)\neq 2$ then
$$\Aut(\overline{\mathcal{M}}_{g,n})\cong\Aut(\overline{M}_{g,n})\cong S_{n}$$
for any $(g,n)\neq (2,1)$ such that $2g-2+n\geq 3$ and $n\geq 1$. Finally, over any field of characteristic zero we have that $\Aut(\overline{\mathcal{M}}_{2,1})\cong\Aut(\overline{M}_{2,1})$ is trivial, and $\Aut(\overline{\mathcal{M}}_{g})\cong\Aut(\overline{M}_{g})$ is trivial for any $g\geq 2$.
\end{Theorem}
\begin{proof}
Let us begin with the genus zero case. If $K$ is any field of characteristic zero the statement follows from Theorem \ref{mella} and the first part of Proposition \ref{rigidautoscheme}. If $K$ is any field of positive characteristic then by Theorem \ref{rig} $\overline{M}_{0,n}^K$ is rigid and $H^0(\overline{M}_{0,n}^K,T_{\overline{M}_{0,n}^K})$ for $n\geq 5$. Now, to conclude it is enough to apply the second part of Proposition \ref{rigidautoscheme}.\\ 
The statement for $(g,n)\neq (2,1)$, $2g-2+n\geq 3$ and $n\geq 1$ follows by Propositions \ref{rigidautoscheme}, \ref{autstack} together with Theorem \ref{alex}. Finally, Propositions \ref{rigidautoscheme}, \ref{autstack} and Remark \ref{n0} yield the result for $(g,n)= (2,1)$ and $g\geq 2, n =0$.
\end{proof}

By \cite[Proposition 4.2.8]{Pe} any group scheme over a field of characteristic zero is reduced. However, in positive characteristic a group scheme is not necessarily reduced. For instance in \cite{Li} there are examples of non-reduced Picard schemes. By Theorem \ref{aut1} we know that over any field of positive characteristic $\Aut(\overline{M}_{0,n})\cong S_{n}$ for any $n\geq 5$. A priori this is not enough to conclude that $\Aut(\overline{M}_{0,n})$ is reduced. Anyway we have the following proposition.

\begin{Proposition}\label{propred}
Let $K$ be a field of positive characteristic. Then $\Aut(\overline{M}_{0,n})$ is reduced.
\end{Proposition}
\begin{proof}
The tangent space of $\Aut(\overline{M}_{0,n})$ at the identity is given by
$$T_{[Id]}\Aut(\overline{M}_{0,n})\cong H^{0}(\overline{M}_{0,n},T_{\overline{M}_{0,n}})$$
If $n\geq 5$, by Theorem \ref{rig} we have $\dim(T_{[Id]}\Aut(\overline{M}_{0,n})) = h^{0}(\overline{M}_{0,n},T_{\overline{M}_{0,n}}) = 0$. So $\Aut(\overline{M}_{0,n})$ is reduced. If $n = 4$ then $\Aut(\overline{M}_{0,4}) \cong PGL(2)$ which is reduced as well.
\end{proof}

\appendix
\section{Automorphisms of $\overline{M}_{g,n}$ over algebraically closed fields}\label{app}
In this appendix we essentially extend the main result of \cite{Ma} on the automorphisms of $\overline{M}_{g,n}$ and $\overline{\mathcal{M}}_{g,n}$ working on an algebraically closed field $K$ with $\ch(K)\neq 2$.\\
Let us begin by discussing the case $g = 0$. In \cite{BM} A. Bruno and M. Mella, thanks to Kapranov's works \cite{Ka2}, managed to translate issues on the moduli space $\overline{M}_{0,n}$ in terms of classical projective geometry of $\mathbb{P}^{n-3}$. Studying linear systems on $\mathbb{P}^{n-3}$ with particular base loci they derived a theorem on the automorphisms of $\overline{M}_{0,n}$ over an algebraically closed field of characteristic zero.
\begin{Theorem}\cite[Theorem 3]{BM}\label{mella}
Let $K$ be an algebraically closed field of characteristic zero. Then the automorphism group of $\overline{M}_{0,n}$ is isomorphic to $S_{n}$ for any $n\geq 5$.
\end{Theorem}
In \cite{Ma} a similar result is proven for $\overline{M}_{g,n}$ over the field $\mathbb{C}$ of complex numbers. However, the proof of \cite[Theorem 3.10]{Ma} works over any algebraically closed field $K$ with $\ch(K)\neq 2$. This is just because \cite[Theorem 0.9]{GKM} works over any algebraically closed field of characteristic different from two and a general $n$-pointed genus $g$ is automorphism-free if $2g-2+n\geq 3$.
\begin{Theorem}\label{alex}
Let $K$ be an algebraically closed field with $\ch(K)\neq 2$. If $2g-2+n\geq 3$, $g\geq 1$ and $(g,n)\neq (2,1)$ then 
$$\Aut(\overline{\mathcal{M}}_{g,n})\cong\Aut(\overline{M}_{g,n})\cong S_{n}$$
\end{Theorem}

\begin{Remark}\label{n0}
By \cite[Proposition 3.6, Theorem 4.4]{Ma}, if $K$ is algebraically closed of characteristic zero we have that $\Aut(\overline{M}_{2,1})\cong \Aut(\overline{M}_{2,1})$ is trivial. Furthermore, by \cite[Proposition 3.5, Theorem 4.4]{Ma} we have that $\Aut(\overline{M}_{g})\cong \Aut(\overline{M}_{g})$ is trivial for any $g\geq 2$. The proofs are based on Royden's theorem \cite[Theorem 6.1]{Moc} which works just over an algebraically closed of characteristic zero .
\end{Remark}

\begin{Remark}\label{m1112}
Let $K$ be an algebraically closed field with $\ch(K)\neq 2,3$. By \cite[Proposition 4.5]{Ma} we have $\Aut(\overline{\mathcal{M}}_{1,2})$ is trivial. Indeed, if $\ch(K)=0$ by \cite[Theorem 2.1]{Hac} $T_{[Id]}\Aut(\overline{\mathcal{M}}_{1,2})\cong H^0(\overline{\mathcal{M}}_{1,2},T_{\overline{\mathcal{M}}_{1,2}})=0$. Finally, since $\overline{M}_{1,1}\cong \mathbb{P}^1$ and $\overline{\mathcal{M}}_{1,1}\cong \mathbb{P}(4,6)$ we have $\Aut(\overline{M}_{1,1})\cong PGL(2)$, while $\Aut(\overline{\mathcal{M}}_{1,1})\cong K^{*}$.
\end{Remark}

\subsubsection*{Acknowledgments}
This work was done while the second named author was a Post-Doctorate at IMPA, funded by CAPES-Brazil. We would like to thank \textit{Aise Johan de Jong} and \textit{Massimiliano Mella} for bringing this subject to our attention and for many helpful comments, \textit{Paul Hacking} for pointing us out \cite{DI}, and \textit{Nicola Pagani} for useful discussion about the singularities of coarse moduli spaces.


\begin{thebibliography}{999999}
\bibitem[Ar]{Ar} \bibaut{M. Artin}, \textit{Lectures on deformations of singularities}, Bombay, Tata Institute, 1976.
\bibitem[BM]{BM} \bibaut{A. Bruno, M. Mella}, \textit{The automorphism group of $\overline{M}_{0,n}$}, J. Eur. Math. Soc. Volume 15, Issue 3, 2013, 949-968.
\bibitem[Co]{Co} \bibaut{M. Cornalba}, \textit{On the locus of curves with automorphisms}, Annali di Matematica Pura ed Applicata, 1987, 149, Issue 1, 135-151.
\bibitem[CLS]{CLS} \bibaut{D. A. Cox, J. B. Little, H. K. Schenck}, \textit{Toric Varieties}, American Mathematical Soc, Graduate studies in mathematics, 2011.
\bibitem[DI]{DI} \bibaut{P. Deligne, L. Illusie}, \textit{Rel\`evements modulo $p^{2}$ et d\'ecomposition du complexe de de Rham}, Invent. Math. 89, 1987, no. 2, 247-270. 
\bibitem[FAG]{FAG} \bibaut{B. Fantechi, L. G\"ottsche , L. Illusie, S. Kleiman, N. Nitsure, A. Vistoli}, \textit{Fundamental algebraic geometry. Grothendieck's FGA explained}, Mathematical Surveys and Monographs, no. 123. American Mathematical Society, Providence, RI, 2005.
\bibitem[Fan]{Fan} \bibaut{B. Fantechi}, \textit{Deformation of Hilbert Schemes of Points on a Surface}, Compositio Mathematica 98, 1995, 205-217.
\bibitem[FMN]{FMN} \bibaut{B. Fantechi, E. Mann, F. Nironi}, \textit{Smooth toric DM stacks}, J. Reine Angew. Math. 648, 2010, 201-244.
\bibitem[GKM]{GKM} \bibaut{A. Gibney, S. Keel, I. Morrison}, \textit{Towards the ample cone of $\overline{M}_{g,n}$}, J. Amer. Math. Soc. 15, 2002, no. 2, 273-294. 
\bibitem[Gro]{Gro} \bibaut{A. Grothendieck}, \textit{Cohomologie locale des faisceaux coh\'erents et th\'eor\`emes de Lefschetz locaux et globaux}, SGA2, North Holland, 1968. 
\bibitem[Hac]{Hac} \bibaut{P. Hacking}, \textit{The moduli space of curves is rigid}, Algebra Number Theory 2, 2008, no. 7, 809-818.
\bibitem[Har]{Har} \bibaut{R. Hartshorne}, \textit{Algebraic geometry}, Graduate Texts in Mathematics, no. 52, Springer-Verlag, New York-Heidelberg, 1977.
\bibitem[Hay]{Hay} \bibaut{T. Haykawa}, \textit{Blowing Ups of 3-dimensional Terminal Singularities}, Publ RIMS, Kyoto Univ. 35, 1999, 515-570.
\bibitem[Ill]{Ill} \bibaut{L. Illusie}, \textit{Complexe Cotangent et D\'eformations I}, Springer Lecture Notes 239, Springer, 1971.
\bibitem[Ka]{Ka} \bibaut{M. Kapranov}, \textit{Deformations of moduli spaces}, 1997. Unpublished manuscript.
\bibitem[Ka2]{Ka2} \bibaut{M. Kapranov}, \textit{Veronese curves and Grothendieck-Knudsen moduli space $\overline{M}_{0,n}$}, J. Algebraic Geom. 2, 1993, no. 2, 239-262.
\bibitem[Ke]{Ke} \bibaut{S. Keel}, \textit{Basepoint freeness for nef and big line bundles in positive characteristic}, Ann. of Math. 2, 149, 1999, no. 1, 253-286.
\bibitem[KM]{KM} \bibaut{S. Keel, S. Mori}, \textit{Quotients by groupoids}, Ann. of Math. 2, 145, 1997, 193-213.
\bibitem[Kn]{Kn2} \bibaut{F. F. Knudsen}, \textit{The projectivity of the moduli space of stable curves II. The stacks $M_{g,n}$}, Math. Scand. 52, 1983, no. 2, 161-199.
\bibitem[Li]{Li} \bibaut{C. Liedtke}, \textit{A note on non-reduced Picard schemes}, J. Pure Appl. Algebra, 213, 2009, no. 5, 737-741.
\bibitem[Li1]{Li1} \bibaut{V. Lin}, \textit{Algebraic functions, configuration spaces, Teichm\"uller spaces, and new holomorphically combinatorial invariants}, Funct. Anal. Appl, 45, 2011, 204-224.
\bibitem[Li2]{Li2} \bibaut{V. Lin}, \textit{Configuration spaces of $\mathbb{C}$ and $\mathbb{CP}^1$: some analytic properties}, \arXiv{math/0403120v3}.
\bibitem[Ma]{Ma} \bibaut{A. Massarenti}, \textit{The automorphism group of $\overline{M}_{g,n}$}, J. London Math. Soc, 2014, 89, 131-150.
\bibitem[MM]{MM} \bibaut{A. Massarenti, M. Mella}, \textit{On the automorphisms of Hassett's moduli spaces}, \arXiv{1307.6828}.
\bibitem[Moc]{Moc}\bibaut{S. Mochizuki}, \textit{Correspondences on hyperbolic curves}, J. Pure Applied Algebra, 131, 1998, 227-244.
\bibitem[Mor]{Mor}\bibaut{A. Moriwaki}, \textit{The $\mathbb{Q}$-Picard group of the moduli space of curves in positive characteristic}, International J. Math, 12, 2001, 519-534.
\bibitem[OS]{OS} \bibaut{S. Okawa, T. Sano}, \textit{Noncommutative rigidity of the moduli stack of stable pointed curves}, \arXiv{1412.7060v1}.
\bibitem[Pa]{Pa} \bibaut{N. Pagani}, \textit{Harer stability and orbifold cohomology}, Pacific J. of Math. 267, 2014, no.2, 465-477.
\bibitem[Pe]{Pe}\bibaut{D. Perrin}, \textit{Approximation des sch$\acute{e}$mas en groupes, quasi compacts sur un corps}, Bull. Soc. Math. France, 104, 1976, no. 3, 323-335. 
\bibitem[Pr]{Pr}\bibaut{Y. G. Prokhorov}, \textit{Lectures on complements on log surfaces}, MSJ Memoirs, 10, Mathematical Society of Japan, Tokyo, 2001.
\bibitem[Re]{Re} \bibaut{M. Reid}, \textit{Young person's guide to canonical singularities}, Proc. Sympos. Pure Math, 46, Providence, R.I: American Mathematical Society, 345-414.
\bibitem[Ro]{Ro} \bibaut{H.L. Royden}, \textit{Automorphisms and isometries of Teichm\"uller spaces}, Advances in the theory of Riemann surfaces Ed. by L. V. Ahlfors, L. Bers, H. M. Farkas, R. C. Gunning, I. Kra, H. E. Rauch, Annals of Math. Studies No.66 (1971), 369-383.
\bibitem[Se]{Se} \bibaut{E. Sernesi}, \textit{Deformations of algebraic schemes}, Grundlehren der Mathematischen Wissenschaften, Fundamental Principles of Mathematical Sciences, 334.
\bibitem[Wi]{Wi}\bibaut{E. Witt}, \textit{Zyklische K\"orper und Algebren der Characteristik $p$ vom Grad $p^{n}$. Struktur diskret bewerteter perfekter K\"orper mit vollkommenem Restklassenk\"orper der Charakteristik $p^{n}$}, J. Reine Angew. Math, 176, 1936, 126-140.
\end{thebibliography}
\end{document}